\documentclass[a4paper,12pt]{article}
\usepackage[cp1251]{inputenc}
\usepackage[russian]{babel}
\usepackage{amsfonts, amssymb, amsmath, amsthm, amscd}
\usepackage{cite}
\author{A.A. Vasil'eva}
\title{Estimates for norms of two-weighted summation operators
on a tree under some conditions on weights}
\date{}
\begin{document}

\maketitle

\newenvironment{Biblio}{%
                  \renewcommand{\refname}{\footnotesize REFERENCES}%
                  \begin{thebibliography}{99}%
                  \itemsep -0.2em}%
                  {\end{thebibliography}}

\def\inff{\mathop{\smash\inf\vphantom\sup}}
\renewcommand{\le}{\leqslant}
\renewcommand{\ge}{\geqslant}
\newcommand{\sgn}{\mathrm {sgn}\,}
\newcommand{\inter}{\mathrm {int}\,}
\newcommand{\dist}{\mathrm {dist}}
\newcommand{\supp}{\mathrm {supp}\,}
\newcommand{\R}{\mathbb{R}}
\renewcommand{\C}{\mathbb{C}}
\newcommand{\Z}{\mathbb{Z}}
\newcommand{\N}{\mathbb{N}}
\newcommand{\Q}{\mathbb{Q}}
\theoremstyle{plain}
\newtheorem{Trm}{Theorem}
\newtheorem{trma}{Theorem}
\newtheorem{Def}{Definition}
\newtheorem{Cor}{Corollary}
\newtheorem{Lem}{Lemma}
\newtheorem{Rem}{Remark}
\newtheorem{Sta}{Proposition}
\renewcommand{\proofname}{\bf Proof}
\renewcommand{\thetrma}{\Alph{trma}}

\section{Introduction}

In this paper, estimates for norms of weighted summation operators
(discrete Hardy-type operators) on a tree were obtained for some
conditions on weights.

The inequalities
\begin{align}
\label{har_dy} \left(\sum \limits _{k=0}^\infty w_k^q\left( \sum
\limits _{j=0}^k u_jf_j\right)^q\right)^{\frac 1q} \le C\left(\sum
\limits _{k=0}^\infty |f_k|^p\right)^{\frac 1p}, \quad (f_k)_{k\in
\Z_+}\in l_p,
\end{align}
were studied in papers of Leindler \cite{l_leind}, Bennett
\cite{ben1, ben2, bennett_g}, Braverman and Stepanov
\cite{vd_step94}, Goldman \cite{gold_man}. The order estimates of
the minimal constant $C$ in (\ref{har_dy}) were first obtained in
\cite{bennett_g} and \cite{vd_step94} (for $1\le p, \, q\le
\infty$ the upper estimates were proved by Heinig and Andersen
\cite{and_hein, hein1}). The similar problem for two-weighted
integration operators on a semiaxis was solved by Bradley
\cite{j_brad}, Mazya and Rozin \cite{mazya1}. Later, these results
were generalized for matrix operators and integration operators
with different kernels (see, e.g., papers of Heinig and Andersen
\cite{and_hein, hein1}, Stepanov \cite{step90, stepanov1}, Oinarov
\cite{r_oin}, Prokhorov and Stepanov \cite{pr_st}, Stepanov and
Ushakova \cite{st_ush}, Rautian \cite{n_raut}, Farsani
\cite{fars_sm}, Oinarov, Persson and Temirkhanova
\cite{oin_per_tem}, Okpoti, Persson and Wedestig
\cite{okp_per_wed1, okp_per_wed2}, and the books
\cite{kuf_mal_pers, kuf_per, grosse_erd}). In the case $p=q=2$
Naimark and Solomyak \cite{naim_sol} showed that the problem of
estimating the norm of weighted integration operator on a regular
tree with weights depending only on distance from the root can be
reduced to a problem on estimating the norm of some weighted
Hardy-type operator on a half-axis.

The criterion of boundedness of a two-weighted integration
operator on a metric tree and order estimates for its norm were
obtained by Evans, Harris and Pick \cite{ev_har_pick}. The
estimate for the norm of a summation operator on a combinatorial
tree can be derived from their result (it will be made in \S
\ref{sect_ev} for $p\le q$). However, this estimate in general
case is rather complicated. Here under some conditions on weights
we obtain estimates which are more simple and convenient for
applications.

The Hardy-type inequalities on trees are used in order to prove
embedding theorems for weighted Sobolev classes on a domain (see
\cite{evans_har, vas_rjmp1, vas_rjmp2}) and in estimating widths
of functional classes, $s$-numbers and entropy numbers of
embedding operators (see \cite{evans_har, e_h_l, ev_har_lang,
lifs_m, l_l, l_l1, solomyak}).

Let $X$, $Y$ be sets, $f_1$, $f_2:\ X\times Y\rightarrow \R_+$. We
write $f_1(x, \, y)\underset{y}{\lesssim} f_2(x, \, y)$ (or
$f_2(x, \, y)\underset{y}{\gtrsim} f_1(x, \, y)$) if, for any
$y\in Y$, there exists $c(y)>0$ such that $f_1(x, \, y)\le
c(y)f_2(x, \, y)$ for each $x\in X$; $f_1(x, \,
y)\underset{y}{\asymp} f_2(x, \, y)$ if $f_1(x, \, y)
\underset{y}{\lesssim} f_2(x, \, y)$ and $f_2(x, \,
y)\underset{y}{\lesssim} f_1(x, \, y)$.

Throughout this paper we consider graphs ${\cal G}$ with finite or
countable vertex set, which will be denoted by ${\bf V}({\cal
G})$. Also we suppose that the graphs have neither multiple edges
nor loops. The set of edges we denote by ${\bf E}({\cal G})$ and
identify pairs of adjacent vertices with edges that connect them.

Let ${\cal T}=({\cal T}, \, \xi_0)$ be a tree rooted at $\xi_0$.
We introduce a partial order on ${\bf V}({\cal T})$ as follows: we
say that $\xi'>\xi$ if there exists a simple path $(\xi_0, \,
\xi_1, \, \dots , \, \xi_n, \, \xi')$ such that $\xi=\xi_k$ for
some $k\in \overline{0, \, n}$; by the distance between $\xi$ and
$\xi'$ we mean the quantity $\rho_{{\cal T}}(\xi, \,
\xi')=\rho_{{\cal T}}(\xi', \, \xi) =n+1-k$. In addition, set
$\rho_{{\cal T}}(\xi, \, \xi)=0$. For $j\in \Z_+$ and $\xi\in {\bf
V}({\cal T})$ write
$$
\label{v1v}{\bf V}_j(\xi):={\bf V}_j ^{{\cal T}}(\xi):=
\{\xi'\ge\xi:\; \rho_{{\cal T}}(\xi, \, \xi')=j\}.
$$
For vertices $\xi\in {\bf V}({\cal T})$, denote by ${\cal
T}_\xi=({\cal T}_\xi, \, \xi)$ the subtree in ${\cal T}$ with
the vertex set
$$
\{\xi'\in {\bf V}({\cal T}):\xi'\ge \xi\}.
$$

Let ${\mathbf W}\subset {\mathbf V}({\mathcal T})$. We say that
${\mathcal G}\subset {\mathcal T}$ is a maximal subgraph on the
set of vertices ${\mathbf W}$ if ${\mathbf V}({\mathcal
G})={\mathbf W}$ and if any two vertices $\xi'$, $\xi''\in
{\mathbf W}$ that are adjacent in ${\mathcal T}$ are also adjacent
in ${\mathcal G}$.

Let ${\cal G}$ be a subgraph in ${\cal T}$. Denote by ${\bf
V}_{\max} ({\cal G})$ and ${\bf V}_{\min}({\cal G})$ the set
of maximal and minimal vertices in ${\cal G}$, respectively. Given
a function $f:{\bf V}({\cal G})\rightarrow \R$, we set
\begin{align}
\label{flpg} \|f\|_{l_p({\cal G})}=\left ( \sum \limits _{\xi \in
{\bf V} ({\cal G})}|f(\xi)|^p\right )^{1/p}.
\end{align}
Denote by $l_p({\cal G})$ the space of functions $f:{\bf
V}({\cal G})\rightarrow \R$ with a finite norm $\|f\|_{l_p({\cal
G})}$.

Let $({\cal G}, \, \xi_0)$ be a disjoint union of trees,
$1\le p\le \infty$, and let $u$, $w:{\bf V}({\cal G})\rightarrow (0, \,
\infty)$ be weight functions. Define the summation operator $S_{u,w,{\cal
G}}$ by
$$
S_{u,w,{\cal G}}f(\xi) = w(\xi)\sum \limits _{\xi'\le \xi}
u(\xi')f(\xi'), \quad \xi \in {\bf V}({\cal G}), \quad f:{\bf
V}({\cal G}) \rightarrow \R.
$$
By $\mathfrak{S}^{p,q}_{{\cal G},u,w}$ we denote the operator norm
of $S_{u,w,{\cal G}}:l_p({\cal G}) \rightarrow l_q({\cal
G})$, i.e., the minimal constant $C$ in the inequality
$$
\left(\sum \limits_{\xi \in {\bf V}({\cal G})} w^q(\xi) \left(
\sum \limits _{\xi'\le \xi}u(\xi')f(\xi')\right)^q\right)^{1/q}
\le C\left(\sum \limits_{\xi \in {\bf V}({\cal G})}
|f(\xi)|^p\right)^{1/p}, \;\; f:{\bf V}({\cal G})\rightarrow \R.
$$

Let us formulate the main results of this paper.
\begin{Trm}
\label{main_plq} Let ${\cal A}$ be a tree and let $1<p<q<\infty$.
Suppose that there are $K\ge 1$, $l_0\in \N$ and $\lambda\in (0,
\, 1)$ such that
\begin{align}
\label{cv1k} {\rm card}\, {\bf V}_1^{\cal A}(\xi)\le K, \quad \xi
\in {\bf V}({\cal A}),
\end{align}
\begin{align}
\label{wq_lam} \frac{u(\xi')}{u(\xi)}\le K,  \quad
\frac{\|w\|_{l_q({\cal A}_{\xi''})}} {\|w\|_{l_q({\cal A}_\xi)}}
\le \lambda, \quad \xi \in {\bf V}({\cal A}), \quad \xi'\in {\bf
V}_1^{\cal A}(\xi), \quad \xi''\in {\bf V}_{l_0}^{\cal A}(\xi).
\end{align}
Then $\mathfrak{S}^{p,q}_{{\cal A},u,w}
\underset{K,\lambda,l_0,p,q}{\asymp} \sup _{\xi \in {\bf V}({\cal A})}
u(\xi)\|w\|_{l_q({\cal A}_\xi)}$.
\end{Trm}

Let $N\in \N\cup \{+\infty\}$, and let $({\cal A}, \, \xi_0)$ be a
tree such that ${\bf V}_{\max}({\cal A}) ={\bf V}^{\cal
A}_N(\xi_0)$. Suppose that there exist a non-decreasing function
$\psi:\R_+\rightarrow \R_+$ and a constant $C_*\ge 1$ such that
$\psi(0)=0$ and for any $0\le j\le j'< N+1$, $\xi\in {\bf
V}_{j}^{\cal A}(\xi_0)$
\begin{align}
\label{cvjj0} C_*^{-1} \cdot 2^{\psi(j')-\psi(j)} \le {\rm
card}\, {\bf V}_{j'-j}^{\cal A}(\xi) \le C_*\cdot
2^{\psi(j')-\psi(j)}.
\end{align}
Let $u$, $w:{\bf V}({\cal A}) \rightarrow (0, \, \infty)$,
$u(\xi)=u_j$, $w(\xi)=w_j$ for $\xi \in {\bf V}_j ^{\cal
A}(\xi_0)$, $1\le q\le p\le \infty$. Estimate the value
$\mathfrak{S}^{p,q}_{{\cal A},u,w}$.

Denote by $\overline{\mathfrak{S}}^{p,q}_{{\cal A},u,w}$
the minimal constant $C$ in the inequality
$$
\left(\sum \limits _{\xi \in {\bf V}({\cal A})} w^q(\xi)
\left(\sum \limits _{\xi'\le \xi} u(\xi')f(\xi')\right)^q
\right)^{\frac 1q} \le C\|f\| _{l_p({\cal A})},
$$
$$f:{\bf V}({\cal A}) \rightarrow \R_+, \quad f(\xi)=f_j \quad\text{for any}
\quad \xi \in {\bf V}_j^{\cal A}(\xi_0),\quad 0\le j< N+1.$$ For
such functions $f$ we have
$$
\left(\sum \limits _{\xi \in {\bf V}({\cal A})} w^q(\xi)
\left(\sum \limits _{\xi'\le \xi} u(\xi')f(\xi')\right)^q
\right)^{\frac 1q} \stackrel{(\ref{cvjj0})} {\underset{C_*}
{\asymp}} \left(\sum \limits _{j=0}^N w_j^q\cdot 2^{\psi(j)}\left(
\sum \limits _{i=0}^j u_if_i\right)^q\right)^{\frac 1q},
$$
$$
\|f\|_{l_p({\cal A})}\stackrel{(\ref{cvjj0})} {\underset{C_*}
{\asymp}} \left( \sum \limits _{j=0}^N f_j^p\cdot
2^{\psi(j)}\right)^{\frac 1p}
$$
(if $C_*=1$, then we have exact equalities). Set $\hat
w_j=w_j\cdot 2^{\frac{\psi(j)}{q}}$, $\hat u_j=u_j\cdot
2^{-\frac{\psi(j)}{p}}$, $0\le j< N+1$. Then
$\overline{\mathfrak{S}}^{p,q}_{{\cal
A},u,w}\underset{C_*}{\asymp} \mathfrak{S}^{p,q}_{\hat u,\hat w}$,
where $\mathfrak{S}^{p,q}_{\hat u,\hat w}$ is the minimal constant
in the inequality
$$
\left(\sum \limits _{j=0}^N \hat w_j^q\left( \sum \limits _{i=0}^j
\hat u_j \varphi_j\right)^q\right)^{\frac 1q} \le C\left(\sum
\limits _{j=0}^N \varphi_j^p\right)^{\frac 1p}, \quad \varphi_j\ge
0, \quad 0\le j< N+1.
$$
Moreover, if $C_*=1$, then we have the exact equality.

\begin{Trm}
\label{p_ge_q} Let $p\ge q$. Then $\mathfrak{S}^{p,q}_{{\cal
A},u,w} \underset{p,q,C_*}{\asymp} \mathfrak{S}^{p,q}_{\hat u,\hat
w}$; if $C_*=1$, then $\mathfrak{S}^{p,q}_{{\cal A},u,w} =
\mathfrak{S}^{p,q}_{\hat u,\hat w}$.
\end{Trm}

The estimates of $\mathfrak{S}^{p,q}_{\hat u,\hat w}$ were
obtained by Heinig, Andersen and Bennett \cite{and_hein, hein1,
bennett_g}.

\begin{trma}
\label{hardy_diskr} Let $1\le p, \, q\le \infty$, and let $\{u_n\}_{n\in
\Z_+}$, $\{w_n\}_{n\in \Z_+}$ be non-negative sequences such that
$$
M_{u,w}:=\sup _{m\in \Z_+}\Bigl(\sum \limits _{n=m}^\infty
w_n^q\Bigr)^{\frac 1q}\Bigl( \sum \limits _{n=0}^m
u_n^{p'}\Bigr)^{\frac{1}{p'}}<\infty\quad\text{for}\quad 1<p\le
q<\infty,
$$
$$
M_{u,w}:=\left(\sum \limits _{m=0}^\infty \left(\Bigl(\sum \limits
_{n=m}^\infty w_n^q\Bigr)^{\frac 1p}\Bigl(\sum \limits _{n=0}^m
u_n^{p'}\Bigr)^{\frac{1}{p'}}\right)
^{\frac{pq}{p-q}}w_m^q\right)^{\frac 1q-\frac
1p}<\infty\quad\text{for}\quad 1\le q<p\le \infty.
$$
Let $\mathfrak{S}^{p,q}_{u,w}$ be the minimal constant in the
inequality
$$
\left(\sum \limits _{n=0}^\infty\left|w_n\sum \limits _{k=0}^n
u_kf_k\right|^q\right)^{1/q}\le C\left( \sum \limits _{n\in
\Z_+}|f_n|^p\right)^{1/p}, \quad \{f_n\}_{n\in\Z_+}\in l_p.
$$
Then $\mathfrak{S}^{p,q}_{u,w} \underset{p,q}{\asymp} M_{u,w}$.
\end{trma}

\section{The discrete analogue of Evans -- Harris -- Pick theorem}
\label{sect_ev}

Let $(\cal{T}, \, \xi_*)$ be a tree, let $\Delta:{\bf
E}({\cal{T}})\rightarrow 2^{\R}$ be a mapping such that for any
$\lambda\in {\bf E}(\cal{T})$ the set $\Delta(\lambda)=[a_\lambda,
\, b_\lambda]$ is a non-degenerate segment. By a metric tree we
mean
$$
\mathbb{T}=({\cal T}, \, \Delta)=\{(t, \, \lambda):\, t\in
[a_\lambda, \, b_\lambda], \; \lambda\in {\bf E}({\cal{T}})\};
$$
here we suppose that if $\xi'\in {\bf V}_1(\xi)$, $\xi''\in {\bf
V}_1(\xi')$, $\lambda=(\xi, \, \xi')$, $\lambda'=(\xi', \,
\xi'')$, then $(b_\lambda, \, \lambda)=(a_{\lambda'}, \,
\lambda')$. The distance between points of $\mathbb{T}$ is defined
as follows: if $(\xi_0, \, \xi_1, \, \dots, \, \xi_n)$ is a simple
path in the tree $\mathcal{T}$, $n\ge 2$, $\lambda_i=(\xi_{i-1},
\, \xi_i)$, $x=(t_1, \, \lambda_1)$, $y=(t_n, \, \lambda_n)$, then
we set
\begin{align}
\label{yx_tt} |y-x|_{\mathbb{T}}=|b_{\lambda_1}-t_1|+\sum \limits
_{i=2}^{n-1} |b_{\lambda_i}-a_{\lambda_i}| +|t_n-a_{\lambda_n}|;
\end{align}
if $x=(t', \, \lambda)$, $y=(t'', \, \lambda)$, then
$|y-x|_{\mathbb{T}}=|t'-t''|$.

We say that $\lambda<\lambda'$ if $\lambda=(\omega, \, \xi)$,
$\lambda'=(\omega', \, \xi')$ and $\xi\le \omega'$; $(t', \,
\lambda')\le(t'', \, \lambda'')$ if $\lambda'<\lambda''$ or
$\lambda'=\lambda''$, $t'\le t''$. If $(t', \, \lambda')\le
(t'', \, \lambda'')$ and $(t', \, \lambda') \ne (t'', \,
\lambda'')$, then we write $(t', \, \lambda')<(t'', \,
\lambda'')$. For $a$, $x\in \mathbb{T}$, $a\le x$, we set
$[a, \, x]=\{y\in \mathbb{T}:\, a\le y\le x\}$.

Let $A_\lambda\subset \Delta(\lambda)$, $\lambda\in {\bf
E}(\mathcal{T})$. We say that the subset $\mathbb{A}=\{(t, \,
\lambda):\; \lambda\in {\bf E}(\mathcal{T}), \; t\in A_\lambda\}$
is measurable if $A_\lambda$ is measurable for any $\lambda\in
{\bf E}(\mathcal{T})$. Its Lebesgue measure is defined by
$$
{\rm mes}\, \mathbb{A}=\sum \limits _{\lambda\in {\bf
E}(\mathcal{T})} {\rm mes}\, A_\lambda.
$$

Let $f:\mathbb{A}\rightarrow \R$. For $\lambda\in {\bf
E}(\mathcal{T})$ define the function $f_\lambda:A_\lambda
\rightarrow \R$ by $f_\lambda(t)=f(t, \, \lambda)$. We say that
the function $f:\mathbb{A}\rightarrow \R$ is Lebesgue integrable
if $f_\lambda$ is Lebesgue integrable for any $\lambda\in {\bf
E}(\mathcal{T})$ and $\sum \limits _{\lambda\in {\bf
E}(\mathcal{T})} \int \limits _{A_\lambda} |f_\lambda(t)|\,
dt<\infty$, and write
$$
\int \limits _{\mathbb{A}} f(x)\, dx=\sum \limits _{\lambda \in
{\bf E}(\mathcal{T})} \int \limits _{A_\lambda} f_\lambda(t)\, dt.
$$

Let $\mathbb{D}\subset \mathbb{T}$ be a connected set. Denote by
${\cal T}_{\mathbb{D}}$ the maximal subtree in $\cal{T}$ such that
for any $\lambda \in {\bf E}({\cal T}_{\mathbb{D}})$ the set
$\{t\in \Delta(\lambda):\; (t, \, \lambda)\in \mathbb{D}\}$ is a
non-degenerated segment. Define $\Delta_{\mathbb{D}}(\lambda)$ as
follows: $\Delta_{\mathbb{D}}(\lambda)=\{t\in \Delta(\lambda):\;
(t, \, \lambda)\in \mathbb{D}\}$, $\lambda \in {\bf E}({\cal
T}_{\mathbb{D}})$. Then $({\cal T}_{\mathbb{D}}, \,
\Delta_{\mathbb{D}})$ is a metric tree. We identify it with the
set $\mathbb{D}$ and call it a metric subtree of $\mathbb{T}$.

Let $\mathbb{D}$ be a metric subtree in $\mathbb{T}$. We say that
the point $x\in \mathbb{D}$ is maximal (minimal) if $y\in
\mathbb{T}\backslash \mathbb{D}$ for any $y>x$ (for any $y<x$,
respectively). Denote by $\mathbb{D}_{\max}$ the set of maximal
points in $\mathbb{D}$.

Let $\mathbb{T}=({\cal T}, \, \Delta)$ be a  metric tree,
$x_0\in \mathbb{T}$, let $u$, $w:\mathbb{T}\rightarrow \R_+$
be measurable functions. Set $\mathbb{T}_{x_0}=\{x\in \mathbb{T}:\; x\ge
x_0\}$,
\begin{align}
\label{i_uw} I_{u,w,x_0}f(x)=w(x)\int \limits_{[x_0, \,  x]}
u(t)f(t)\, dt, \quad x\in \mathbb{T}_{x_0}.
\end{align}

Suppose that ${\bf V}({\cal T})$ is finite.

\label{j_x0}Denote by ${\cal J}_{x_0}={\cal J}_{x_0}(\mathbb{T})$
the family of metric subtrees $\mathbb{D}\subset \mathbb{T}$
with the following properties:
\begin{enumerate}
\item $x_0$ is a minimal point in $\mathbb{D}$;
\item if $x\in \partial \mathbb{D}\backslash \{x_0\}$, then
$x$ is maximal in $\mathbb{D}$.
\end{enumerate}

{\bf Example.} Let the metric tree $\mathbb{T}=({\cal T}, \, \Delta)$
be defined as follows. The vertex $\xi_0$ is a root of ${\cal T}$,
${\bf V}_1(\xi_0)=\{\xi_1\}$, ${\bf
V}_1(\xi_1)=\{\xi_2, \, \xi_3\}$, $\lambda_1=(\xi_0, \, \xi_1)$,
$\lambda_2=(\xi_1, \, \xi_2)$, $\lambda_3=(\xi_1, \, \xi_3)$,
$\Delta(\lambda_i)=[0, \, 1]$, $i=1, \, 2, \, 3$. Let
$\mathbb{D}=\{(t, \, \lambda_i):\, t\in [0, \, 1], \, i=1, \,
2\}$, $x_0=(0, \, \lambda_1)$. Then ${\cal D}\notin {\cal
J}_{x_0}(\mathbb{T})$. Indeed, $(1, \, \lambda_1)=(0,
\, \lambda_2)$ is a boundary point, but it is not maximal.

Given a subtree $\mathbb{D}\subset \mathbb{T}$, we denote by
$L_p^{{\rm discr}}(\mathbb{D})$ the set of functions $\phi:
\mathbb{D}\rightarrow \R$ that are constants on each edge
of $\mathbb{D}$.

Let $\mathbb{D}\in {\cal J}_{x_0}(\mathbb{T})$, $\partial
\mathbb{D} \backslash \{x_0\}\subset G\subset \mathbb{D}_{\max}$.
We write
\begin{align}
\label{ad} \alpha_{\mathbb{D},G}=\inf
\left\{\|f\|_{L_p(\mathbb{D})}:f\in L_p(\mathbb{D}),\, \int
\limits _{[x_0, \, t]} |f(x)|u(x)\, dx=1\quad \forall t\in
G\right\},
\end{align}
$$
\alpha^{\rm discr}_{\mathbb{D},G}= \inf \left\{
\|f\|_{L_p(\mathbb{D})}:f\in L_p^{\rm discr}(\mathbb{D}),\, \int
\limits _{[x_0, \, t]} |f(x)|u(x)\, dx=1\quad \forall t\in
G\right\},
$$
\begin{align}
\label{a_d} \alpha _{\mathbb{D}}=\alpha_{\mathbb{D},\partial
\mathbb{D} \backslash \{x_0\}}, \quad \alpha _{\mathbb{D}} ^{\rm
discr}=\alpha^{\rm discr}_{\mathbb{D},\partial \mathbb{D}
\backslash \{x_0\}}.
\end{align}

\begin{Rem} \label{rem}
If the function $u$ is a constant on each edge of
$\mathbb{D}$, then $\alpha_{\mathbb{D},G}=\alpha^{\rm
discr}_{\mathbb{D},G}$.
\end{Rem}

The following result was proved by Evans, Harris and Pick
\cite{ev_har_pick}.
\begin{trma}
\label{cr_har} Let $1\le p\le q\le \infty$. Then the operator
$I_{u,w,x_0}:L_p(\mathbb{T}_{x_0})\rightarrow
L_q(\mathbb{T}_{x_0})$ is bounded if and only if
\begin{align}
\label{c_uw} C_{u,w}:=\sup _{\mathbb{D}\in {\cal
J}_{x_0}}\frac{\|w\chi_{\mathbb{T}_{x_0}\backslash \mathbb{D}}\|
_{L_q(\mathbb{T}_{x_0})}}{\alpha _{\mathbb{D}}}<\infty.
\end{align}
Moreover, $C_{u,w}\le
\|I_{u,w}\|_{L_p(\mathbb{T}_{x_0})\rightarrow
L_q(\mathbb{T}_{x_0})}\le 4C_{u,w}$.
\end{trma}

The quantity $\alpha_{\mathbb{D}}$ can be calculated recursively.
The following result is also proved in \cite{ev_har_pick}.
\begin{trma}
\label{rekurs} Let $\mathbb{D}\in {\cal J}_{x_0}$,
$\mathbb{D}=\cup _{j=0}^m \mathbb{D}_j$, $\mathbb{D}_0=[x_0, \,
y_0]$, $x_0<y_0$, $\mathbb{D}_j\in {\cal J}_{y_0}$, $1\le j\le m$,
$\mathbb{D}_i\cap \mathbb{D}_j=\{y_0\}$, $i\ne j$. Then
$$
\frac{1}{\alpha_{\mathbb{D}}}=\left\|\left(\alpha_{\mathbb{D}_0}^{-1},
\, \left\|(\alpha_{\mathbb{D}_i})_{i=1}^m\right\|^{-1}
_{l_p^m}\right)\right\|_{l_{p'}^2}.
$$
\end{trma}

Let ${\bf V}({\cal A})$ be finite. We obtain two-sided estimates
of $\mathfrak{S}^{p,q}_{{\cal A},u,w}$ for $p\le q$.

Let $\hat\xi\in {\bf V}({\cal A})$, ${\cal D}\subset {\cal
A}_{\hat \xi}$, ${\bf V}_{\max}({\cal D})\backslash {\bf V}_{\max}
({\cal A}) \subset \Gamma \subset {\bf V}_{\max}({\cal D})$.
We say that $({\cal D}, \, \Gamma)\in {\cal J}^\circ_{\hat \xi}$
if the following conditions hold:
\begin{enumerate}
\item $\hat \xi$ is minimal in ${\cal D}$,
\item if $\xi\in {\bf V}({\cal D})$ is not maximal in
${\cal D}$, then ${\bf V}_1(\xi)\subset {\bf V}({\cal D})$.
\end{enumerate}
Denote by ${\cal D}_\Gamma$ the subtree of ${\cal D}$ such that
${\bf V}({\cal D}_\Gamma) ={\bf V}({\cal D})\backslash \Gamma$.
Then
\begin{align}
\label{va_x_dg} {\bf V}({\cal A}_{\hat \xi}\backslash {\cal
D}_\Gamma) =\cup _{\xi \in \Gamma} {\bf V}({\cal A}_\xi).
\end{align}
If $({\cal D}, \, \Gamma)\in {\cal J}^\circ_{\hat \xi}$ and
$\Gamma\ne \varnothing$, we write $({\cal D}, \, \Gamma)\in {\cal
J}'_{\hat \xi}$.

For $({\cal D}, \, \Gamma)\in {\cal J}^\circ_{\hat \xi}$ we set
\begin{align}
\label{betad} \beta_{{\cal D},\Gamma}=\inf \left\{ \| \varphi\|
_{l_p({\cal A}_{\hat \xi})}: \, \sum \limits _{\hat \xi\le \xi'\le
\xi} |\varphi(\xi')| u(\xi')=1, \;\; \forall \xi \in
\Gamma\right\}.
\end{align}
Notice that if $\Gamma =\varnothing$, then $\beta_{{\cal
D},\Gamma}=0$.

\begin{Lem}
\label{hardy_cr} Let $p\le q$. Then
$$
\mathfrak{S}^{p,q}_{{\cal A}_{\hat \xi},u,w}\underset{p,q}
{\asymp} \sup _{({\cal D},\,\Gamma)\in {\cal J}'_{\hat \xi}}
\frac{\|w\chi_{{\cal A}_{\hat \xi}\backslash {\cal
D}_\Gamma}\|_{l_q({\cal A}_{\hat \xi})}} {\beta _{{\cal D},\Gamma}}.
$$
\end{Lem}
\begin{proof}
Add a vertex $\xi_*$ to the set ${\bf V}({\cal A})$ and connect it
with $\xi_0$ by an edge. Thus we obtain a tree $({\cal A}^{\#}, \,
\xi_*)$. Define the mapping $\Delta$ by $\Delta(\lambda)=[0, \,
1]$, $\lambda \in \mathbf{E}({\cal A}^{\#})$, and set
${\mathbb{A}}=({\cal A}^{\#}, \, \Delta)$. For each $\xi \in {\bf
V}({\cal A})$ denote by $\lambda_\xi$ an edge of the tree ${\cal
A}^{\#}$ with the head $\xi$ (i.e., $\lambda_\xi=(\xi', \, \xi)$,
$\xi'<\xi$). Given a function $\psi:{\bf V}({\cal A})\rightarrow
\R$, we define a function $\psi^{\#} :{\mathbb{A}}\rightarrow \R$
by $\psi^{\#}(t, \, \lambda_\xi)=\psi(\xi)$, $t\in [0, \, 1]$,
$\xi \in {\bf V}({\cal A})$.

Let $x_0=(0, \, \lambda_{\hat \xi})\in\mathbb{A}$. By
the H\"{o}lder's inequality and Theorem \ref{cr_har},
\begin{align}
\label{sa} \mathfrak{S}_{{\cal A}_{\hat \xi},u,w}
\underset{p,q}{\asymp} \|I_{u^{\#},w^{\#},x_0}\|
_{L_p(\mathbb{A}_{x_0})\rightarrow L_q(\mathbb{A}_{x_0})}
\stackrel{(\ref{c_uw})}{\underset{p,q}{\asymp}} \sup _{\mathbb{D}
\in {\cal J}_{x_0}} \frac{\|w^{\#}\chi_{\mathbb{A}_{x_0}
\backslash
\mathbb{D}}\|_{L_q(\mathbb{A}_{x_0})}}{\alpha_{\mathbb{D}}},
\end{align}
with $\alpha_{\mathbb{D}}$ defined by (\ref{ad}) and (\ref{a_d}).

Let $\mathbb{D}\in {\cal J}_{x_0}$, $\mathbb{D} \ne
\mathbb{A}_{x_0}$, ${\cal D}^{\#}={\cal
A}^{\#}_{\mathbb{D}}$, ${\cal D}=({\cal D}^{\#})_{\hat\xi} \subset
{\cal A}$,
$$
\Gamma =\{\xi \in {\bf V}({\cal D}):\; \exists t\in (0, \, 1]:\;
(t, \, \lambda_\xi)\in \partial \mathbb{D} \backslash \{x_0\}\}.
$$
Then ${\bf V}_{\max}({\cal D}) \backslash {\bf V} _{\max}({\cal
A})\subset \Gamma \subset {\bf V}_{\max}({\cal D})$ and $\Gamma\ne
\varnothing$. Prove that $({\cal D}, \, \Gamma)\in {\cal
J}'_{\hat\xi}$. Property 1 holds by construction. Prove Property
2. Indeed, let $\xi\in {\bf V}({\cal D})$, and suppose that there
are vertices $\xi'\in {\bf V}_1(\xi) \backslash {\bf V}({\cal D})$
and $\xi''\in {\bf V}_1(\xi) \cap {\bf V}({\cal D})$. Let $\eta\in
{\bf V}({\cal A} ^{\#})$ be the direct predecessor of $\xi$. Then
$(1, \, (\eta, \, \xi))=(0, \, (\xi, \, \xi'))= (0, \, (\xi, \,
\xi''))$ is a boundary point in $\mathbb{D}$ and it is not
maximal.

Set $\mathbb{D}^+=({\cal D}^{\#}, \, \Delta)$, $\mathbb{D}^-
=(({\cal D}^{\#})_\Gamma, \, \Delta)$, $G=\mathbb{D}^+_{\max}
\backslash {\rm int}\, \mathbb{D}$.

We have
$$
\|w^{\#}\chi _{\mathbb{A} _{x_0}\backslash
\mathbb{D}}\|_{L_q(\mathbb{A}_{x_0})} \le \|w^{\#}\chi
_{\mathbb{A} _{x_0}\backslash \mathbb{D}^-}\|_{L_q
(\mathbb{A}_{x_0})}= \|w\chi _{{\cal A}_{\hat \xi}\backslash
{\cal D}_\Gamma}\|_{l_q({\cal A}_{\hat \xi})};
$$
by Remark \ref{rem},
$$
\alpha_{\mathbb{D}}\ge \alpha_{\mathbb{D}^+_{\max}, G}=
\alpha_{\mathbb{D}^+_{\max}, G}^{\rm discr}= \beta_{{\cal
D},\Gamma}.
$$
This yields the desired upper estimate for $\mathfrak{S}_{{\cal A}
_{\hat\xi},u,w}$.

Prove the lower estimate. Let $({\cal D}, \, \Gamma)\in {\cal J}'
_{\hat\xi}({\cal A})$. Take a function
$f\in l_p({\cal A}_{\hat \xi})$ such that $\sum \limits
_{\hat \xi\le \xi'\le \xi}|f(\xi')|u(\xi')=1$
for any $\xi\in \Gamma$, $\|f\|_{l_p({\cal A}_{\hat \xi})}=
\beta_{{\cal D},\Gamma}$, $\tilde f =\frac{f}{\beta_{{\cal D},\Gamma}}$.
Then $\tilde f(\xi')=0$ for any $\xi'>\xi$, $\xi\in \Gamma$,
$\|\tilde f\|_{l_p({\cal A}_{\hat \xi})}=1$,
$$
\left(\sum \limits _{\xi \in {\bf V}({\cal A}_{\hat \xi})}
w^q(\xi)\left(\sum \limits _{\hat \xi\le \xi'\le \xi}
u(\xi')|\tilde f(\xi')|\right)^q\right)^{1/q}\ge
$$
$$
\ge\left(\sum \limits _{\xi \in \Gamma} \sum \limits _{\tilde
\xi\ge \xi} w^q(\tilde \xi) \left(\sum \limits _{\hat \xi\le
\xi'\le \xi} u(\xi')|\tilde f(\xi')|\right)^q\right)^{1/q}
\stackrel{(\ref{va_x_dg})}{=} \beta_{{\cal D},\Gamma}^{-1}
\|w\|_{l_q({\cal A}_{\hat \xi} \backslash {\cal D}_\Gamma)}.
$$
This completes the proof.
\end{proof}
\begin{Sta}
\label{cor2} Let $\xi_*\in {\bf V}({\cal A})$, ${\bf V}^{\cal
A}_1(\xi_*)=\{\xi_1, \, \dots, \, \xi _m\}$, $({\cal D}, \,
\Gamma)\in {\cal J}'_{\xi_*}$, ${\cal D}_{j}={\cal D}_{\xi_j}$,
$\Gamma_j=\Gamma \cap {\bf V}({\cal D}_j)$. Then
\begin{align}
\label{bjoin} \beta _{{\cal D},\Gamma}^{-1}=
\left\|\left(u(\xi_*), \, \left\|(\beta_{{\cal D}_j,\Gamma_j})
_{j=1} ^m \right\| ^{-1} _{l_p^m} \right) \right\|_{l^2_{p'}}.
\end{align}
\end{Sta}
This proposition follows from Theorem \ref{rekurs} and Remark
\ref{rem}.

\section{An estimate for the norm of a weighted summation operator
on a tree: case $p<q$}

Let $({\cal A}, \, \xi_0)$ be a tree with a finite vertex set, and let
$u, \, w:{\bf V}({\cal
A}) \rightarrow (0, \, \infty)$.

For $\xi_*\in {\bf V}({\cal A})$ and $({\cal D}, \, \Gamma) \in {\cal
J}'_{\xi_*}$ we set $B_{{\cal D},\Gamma}=\beta _{{\cal
D},\Gamma}^{-1}$.

\begin{Lem}
\label{bd_wq_est} Let $1<p<q<\infty$. Then there is $\sigma_*=\sigma_*(p, \, q)
\in \left(0, \, \frac 18\right)$ such that if
\begin{align}
\label{u_w_lq_1} u(\xi) \|w\| _{l_q({\cal A}_\xi)} \le 1, \quad
\xi\in {\bf V}({\cal A}),
\end{align}
and
\begin{align}
\label{w_lq_w_lq} \frac{\|w\|_{l_q({\cal
A}_{\xi'})}}{\|w\|_{l_q({\cal A}_\xi})} \le \sigma, \quad \xi \in
{\bf V}({\cal A}), \quad \xi' \in {\bf V}_1^{\cal A}(\xi)
\end{align}
with $\sigma\in (0, \, \sigma_*)$, then for any $({\cal D}, \,
\Gamma)\in {\cal J}'_{\xi_*}$
\begin{align}
\label{bd_w_lq_pq_1} B_{{\cal D},\Gamma}\|w\|_{l_q({\cal
A}_{\xi_*}\backslash {\cal D}_\Gamma)}
\underset{p,q}{\lesssim} 1.
\end{align}
\end{Lem}
\begin{proof}
For each $t\in [1, \, \infty]$ we set
\begin{align}
\label{def_f_of_t}
f(t):=\frac{\int \limits_0^t \sigma^{s/3}\, ds}{\int \limits_0^1
\sigma^{s/3}\, ds}.
\end{align}
Suppose that $\xi_*$ is not minimal in ${\cal A}$. Let $\hat \xi$
be the direct predecessor of $\xi_*$. Then $\|w\|_{l_q({\cal
A}_{\xi_*}\backslash {\cal D}_\Gamma)} \stackrel
{(\ref{w_lq_w_lq})} {=} \sigma ^t \|w\| _{l_q({\cal
A}_{\hat\xi})}$, $t\ge 1$. Prove that there is $c=c(p, \, q)\ge 1$
such that
\begin{align}
\label{bd_w_1111} B_{{\cal D},\Gamma}\|w\|_{l_q({\cal A}_{\xi_*}\backslash
{\cal D}_\Gamma)} \le c f^{\frac{1}{p'}}(t).
\end{align}
If $\xi_*$ is minimal, then we prove that there is
$c=c(p, \, q)\ge 1$ such that
\begin{align}
\label{bd_w_2222} B_{{\cal D},\Gamma}\|w\|_{l_q({\cal A}_{\xi_*}\backslash
{\cal D}_\Gamma)} \le c f^{\frac{1}{p'}}(\infty).
\end{align}

Denote by $\nu_{\cal D}$ the maximal length of a path in ${\cal
D}$ with beginning at the point $\xi_*$. We shall prove
(\ref{bd_w_1111}) and (\ref{bd_w_2222}) by induction on $\nu_{\cal D}$.

If $\nu_{\cal D}=0$, then ${\cal D}=\{\xi_*\}$, $\Gamma=
\{\xi_*\}$, $B_{{\cal D},\Gamma}=u(\xi_*)$,
${\cal D}_\Gamma=\varnothing$, ${\cal A}_{\xi_*}
\backslash {\cal D}_\Gamma={\cal A}_{\xi_*}$. Therefore,
(\ref{bd_w_1111}) and (\ref{bd_w_2222}) follow from (\ref{u_w_lq_1})
and from the inequality $f(t)\ge 1$, $t\ge 1$.

Suppose that the assertion is proved for any ${\cal D}$ such that
$\nu_{\cal D}\le \nu$. Prove it for $\nu_{\cal D}=\nu+1$.

Let ${\bf V}_1^{\cal A}(\xi_*)=\{\xi_1, \, \dots, \, \xi_m\}$,
${\cal A}_i={\cal A}_{\xi_i}$, ${\cal D}_i={\cal D}_{\xi_i}$,
$\Gamma_i={\bf V}({\cal D}_i)\cap \Gamma$, ${\cal G}_i={\cal
A}_i\backslash ({\cal D}_i)_{\Gamma _i}$. Then
$$
{\cal A}_{\xi_*}=\{\xi_*\} \cup {\cal A}_1\cup \dots \cup {\cal
A}_m, \quad {\cal D}=\{\xi_*\}\cup {\cal D}_1\cup \dots \cup {\cal
D}_m,
$$
\begin{align}
\label{norm_w_a_g_sum} \|w\|_{l_q({\cal A}_{\xi_*} \backslash
{\cal D} _{\Gamma})} \stackrel{(\ref{va_x_dg})}{=}\left(\sum
\limits _{i=1}^m \|w\|^q_{l_q({\cal G}_i)}\right)^{1/q}.
\end{align}

Set $\alpha_i=\frac{\|w\|_{l_q({\cal G}_i)}}{\|w\|_{l_q({\cal
A}_{\xi_*})}}$. Then
\begin{align}
\label{bdi_wi} \alpha_i=\sigma^{t_i}, \quad t_i\ge 1, \quad
\beta_{{\cal D}_i,\Gamma_i}\ge c^{-1}f^{-\frac{1}{p'}}(t_i)
\|w\|_{l_q({\cal G}_i)}.
\end{align}
Indeed, the first relation follows from (\ref{w_lq_w_lq}). If
$\Gamma_i\ne \varnothing$, then $({\cal D}_i, \, \Gamma_i) \in
{\cal J}'_{\xi_i}$, and the second relation holds by induction
hypothesis. If $\Gamma_i=\varnothing$, then ${\cal D}_i={\cal
A}_i$, $({\cal D}_i)_{\Gamma_i} ={\cal A}_i$ and $\|w\|_{l_q({\cal
G}_i)}=0$.

By (\ref{bjoin}),
$$
B_{{\cal D},\Gamma}^{p'}\|w\|^{p'}_{l_q({\cal A}_{\xi_*}\backslash
{\cal D}_\Gamma)}=u^{p'}(\xi_*)\|w\|^{p'}_{l_q({\cal A}_{\xi_*}\backslash
{\cal D}_\Gamma)}+\left(\sum \limits _{i=1}^m B_{{\cal
D}_i,\Gamma_i}^{-p}\right)^{-\frac{p'}{p}} \|w\|^{p'}_{l_q({\cal
A}_{\xi_*}\backslash {\cal D}_\Gamma)}
\stackrel{(\ref{bdi_wi})}{\le}
$$
$$
\le u^{p'}(\xi_*)\|w\|^{p'}_{l_q({\cal A}_{\xi_*}\backslash {\cal
D}_\Gamma)} + c^{p'} \left(\sum \limits _{i=1}^m
\|w\|^p_{l_q({\cal G}_i)} f^{-\frac{p}{p'}}(t_i)
\right)^{-\frac{p'}{p}} \|w\|^{p'} _{l_q({\cal
A}_{\xi_*}\backslash {\cal
D}_\Gamma)}\stackrel{(\ref{norm_w_a_g_sum})}{=}
$$
$$
= u^{p'}(\xi_*) \|w\|^{p'}_{l_q({\cal A}_{\xi_*})}
\frac{\left(\sum \limits _{i=1}^m \|w\|^q _{l_q({\cal G}_i)}
\right)^{\frac{p'}{q}}}{\|w\|^{p'}_{l_q({\cal A}_{\xi_*})}}+
$$
$$
+c^{p'}\left(\frac{\sum \limits _{i=1}^m \|w\|^q_{l_q({\cal
G}_i)}} {\|w\|^q_{l_q({\cal A}_{\xi_*})}}\right)^{\frac{p'}{q}}
\left( \frac{\sum \limits _{i=1} ^m \|w\|^p_{l_q({\cal G}_i)}
f^{-\frac{p}{p'}}(t_i)}{\|w\|^p_{l_q({\cal
A}_{\xi_*})}}\right)^{-\frac{p'}{p}}
\stackrel{(\ref{u_w_lq_1})}{\le}
$$
$$
\le \left(\sum \limits_{i=1}^m \alpha_i ^q\right)^{\frac{p'}{q}}+
c^{p'}\left(\sum \limits_{i=1}^m \alpha_i ^q\right)^{\frac{p'}{q}}
\left(\sum \limits _{i=1}^m
\alpha_i^pf^{-\frac{p}{p'}}(t_i)\right)^{-\frac{p'}{p}},
$$
i.e.,
\begin{align}
\label{bdg_ai} B_{{\cal D},\Gamma}^{p'}\|w\|^{p'}_{l_q({\cal
A}_{\xi_*}\backslash {\cal D}_\Gamma)}\le \left(\sum
\limits_{i=1}^m \alpha_i ^q\right)^{\frac{p'}{q}}+
c^{p'}\left(\sum \limits_{i=1}^m \alpha_i ^q\right)^{\frac{p'}{q}}
\left(\sum \limits _{i=1}^m
\alpha_i^pf^{-\frac{p}{p'}}(t_i)\right)^{-\frac{p'}{p}}=:S.
\end{align}
Let $t_0=\min _{1\le i\le m} t_i$, $I_1=\{i\in \overline{1, \,
m}:\; t_i=t_0\}$, $I_2=\{1, \, \dots, \, m\}\backslash I_1$. Since
\begin{align}
\label{int} \int \limits _0^s \sigma^{t/3}\, dt=\frac{3}{|\log
\sigma|} (1-\sigma^{s/3}),
\end{align}
it follows that for $1\le i\le m$
$$
\frac{f(t_i)}{f(t_0)} =\frac{1-\sigma
^{t_i/3}}{1-\sigma^{t_0/3}}\le
1+\frac{\sigma^{1/3}}{1-\sigma^{1/3}} \le 1+2\sigma^{1/3}
$$
for $\sigma\le \frac 18$. Hence,
$\left(\frac{f(t_i)}{f(t_0)}\right)^{-\frac{p}{p'}}\ge
(1+2\sigma^{1/3})^{-\frac{p}{p'}}\ge 1-\frac{2p}{p'}\sigma^{1/3}$.
Thus,
\begin{align}
\label{s1s1} S\le \left(\sum \limits_{i=1}^m \alpha_i
^q\right)^{\frac{p'}{q}}+c^{p'}\left(\sum \limits_{i=1}^m \alpha_i
^q\right)^{\frac{p'}{q}}\left(\sum \limits _{i\in I_1} \alpha_i^p
+\sum \limits_{i\in I_2}\alpha_i^p \left(1-\frac{2p}{p'}
\sigma^{\frac 13}\right)\right)^{-\frac{p'}{p}} f(t_0)=:\tilde S.
\end{align}

Estimate $\tilde S$ for small $\sigma$. Since $p<q$, there is
$\varepsilon_0=\varepsilon_0(p, \, q) \in \left(0, \, \frac
13\right)$ such that for any $\varepsilon\in [0, \,
\varepsilon_0]$
\begin{align}
\label{1e} (1-\varepsilon)^{\frac pq}+\frac{\varepsilon ^{\frac
pq}}{2}\ge 1.
\end{align}
Let $\sigma \le \min \left\{\frac 18, \, \left(\frac{p'}{4p}
\right)^3\right\}=:\sigma_1$. Set $\beta=\left(\sum
\limits_{i=1}^m \alpha_i^q\right)^{\frac 1q}$. Then
\begin{align}
\label{beta_le_1} \beta^q =\sum \limits _{i=1}^m \alpha_i^q =\sum
\limits _{i=1}^m \frac{\|w\|^q_{l_q({\cal
G}_i)}}{\|w\|^q_{l_q({\cal A}_{\xi_*})}}
\stackrel{(\ref{norm_w_a_g_sum})}{=} \frac{\|w\|^q _{l_q({\cal
A}_{\xi_*} \backslash {\cal D}_\Gamma)}} {\|w\|^q_{l_q({\cal
A}_{\xi_*})}}\le 1.
\end{align}

First we consider $\{\alpha_i\}_{i=1}^m$ satisfying the following
property: there is $i_*\in \overline{1, \, m}$ such that
$\alpha_{i_*}=\beta(1-\varepsilon)^{1/q}$, $\sum \limits_{i\ne
i_*} \alpha_i^q=\beta^q\varepsilon$, $0\le \varepsilon<\varepsilon
_0$. Recall that $\alpha_i=\sigma^{t_i}$ and for $i\in I_1$, $j\in
I_2$ the inequality $\alpha_i>\alpha_j$ holds. Hence,
$I_1=\{i_*\}$ and
$$
\sum \limits _{i\in I_1} \alpha_i^p +\sum \limits_{i\in I_2}
\alpha_i^p \left(1-\frac{2p}{p'} \sigma^{\frac 13}\right)=
\alpha_{i_*}^p +\sum \limits _{i\ne i_*}\alpha_i^p
\left(1-\frac{2p}{p'} \sigma^{\frac 13}\right)\ge
$$
$$
\ge \alpha_{i_*}^p +\left(\sum \limits_{i\ne i_*} \alpha_i ^q
\right)^{\frac pq} \left(1-\frac{2p}{p'} \sigma^{\frac
13}\right)=
$$
$$
= \beta^p(1-\varepsilon)^{\frac pq}
+\beta^p\varepsilon ^{\frac pq} \left(1-\frac{2p}{p'}
\sigma^{\frac 13}\right)\ge \beta^p \left((1-\varepsilon)^{\frac
pq}+\frac{\varepsilon ^{\frac pq}}{2}\right)
\stackrel{(\ref{1e})}{\ge} \beta^{p}.
$$
Thus,
\begin{align}
\label{til_s_le} \tilde S\le \beta^{p'}+c^{p'}f(t_0).
\end{align}

From (\ref{beta_le_1}) it follows that
\begin{align}
\label{beta_sigma_st}
\beta=\sigma^{t_*}, \quad t_*\ge 0.
\end{align}
Show that there is $\sigma_2=\sigma_2(p)\in (0, \, \sigma_1)$
such that for $0<\sigma\le \sigma_2$, $c\ge 2$
\begin{align}
\label{bp_cft0} \beta^{p'}+c^{p'}f(t_0)\le c^{p'}f(t_*+1),
\end{align}
i.e.,
\begin{align}
\label{bp_cft0_1}
\sigma^{p't_*}\le c^{p'}(f(t_*+1)-f(t_0)).
\end{align}
Indeed,
$\sigma^{t_0}=\alpha_{i_*}=\beta(1-\varepsilon)^{\frac 1q}=
\sigma^{t_*}(1-\varepsilon)^{\frac 1q}$. Let $t_0=t_*+\kappa$.
Then $(1-\varepsilon)^{\frac 1q}=\sigma^{\kappa}$. Since
$(1-\varepsilon)^{\frac 1q}\ge \frac 23$, for small $\sigma$
we get $\kappa \le \frac 14$. Hence, $t_0\le t_*+\frac 14$ and
$t_*\ge t_0-\frac 14\ge \frac 34$. Therefore,
$$
f(t_*+1)-f(t_0)\ge f(t_*+1)-f\left(t_*+\frac 14\right)
\stackrel{(\ref{def_f_of_t}),(\ref{int})}{=} \frac{1-\sigma^{\frac
14}}{1-\sigma^{\frac 13}}\sigma^{\frac{t_*}{3}+\frac{1}{12}}.
$$
If $\sigma$ is sufficiently small, then
$c^{p'}\frac{1-\sigma^{1/4}} {1-\sigma^{1/3}}\ge 1$. Thus, in
order to prove (\ref{bp_cft0_1}) it is sufficient to check that
$p't_*\ge \frac{t_*}{3}+\frac {1}{12}$. Indeed, it follows from
the inequalities $t_*\ge \frac 34$ and $\frac 34>\frac
14+\frac{1}{12}$.

This completes the proof of (\ref{bp_cft0}). If $\xi_*$ is a
minimal vertex, then (\ref{bdg_ai}), (\ref{s1s1}),
(\ref{til_s_le}) and (\ref{bp_cft0}) yield that $B_{{\cal
D},\Gamma}^{p'}\|w\|^{p'}_{l_q({\cal A}_{\xi_*}\backslash {\cal
D}_\Gamma)} \le c^{p'}f(t_*+1) \le c^{p'}f(\infty)$, which implies
(\ref{bd_w_2222}). Suppose that the vertex $\xi_*$ in not minimal.
Let $\hat\xi$ be the direct predecessor of $\xi_*$. Then
$\frac{\|w\|_{l_q({\cal A}_{\xi_*})}} {\|w\|_{l_q({\cal A}_{\hat
\xi})}}\le \sigma$ by (\ref{w_lq_w_lq}). Therefore,
$$
\frac{\|w\|_{l_q({\cal A}_{\xi_*}\backslash {\cal D}_\Gamma)}}
{\|w\|_{l_q({\cal A}_{\hat \xi})}} \stackrel
{(\ref{norm_w_a_g_sum})} {=} \frac{\left(\sum \limits _{i=1}^m
\|w\|^q_{l_q({\cal G}_i)}\right)^{\frac 1q} }{\|w\|_{l_q({\cal
A}_{\xi_*})}}\cdot \frac{\|w\|_{l_q({\cal A}_{\xi _*}
)}}{\|w\|_{l_q({\cal A}_{\hat \xi})}}\le \left(\sum
\limits_{i=1}^m \alpha_i ^q\right)^{\frac{1}{q}}\sigma
\stackrel{(\ref{beta_sigma_st})} {=}\sigma^{t_*+1},
$$
i.e., $\|w\|_{l_q({\cal A}_{\xi_*}\backslash {\cal
D}_\Gamma)}=\|w\|_{l_q({\cal A}_{\hat \xi})} \sigma^t$, $t\ge
t_*+1$. This together with (\ref{bdg_ai}), (\ref{s1s1}),
(\ref{til_s_le}) and (\ref{bp_cft0}) yields that $B_{{\cal
D},\Gamma}^{p'}\|w\|^{p'}_{l_q({\cal A}_{\xi_*}\backslash {\cal
D}_\Gamma)} \le c^{p'}f(t_*+1) \le c^{p'}f(t)$, which implies
(\ref{bd_w_1111}).

Let, now, for any $i\in \overline{1, \, m}$ the inequality
$\alpha_i\le \beta(1-\varepsilon_0)^{\frac 1q}$ holds. Prove that
there is $a=a(p, \, q)<1$ such that
\begin{align}
\label{aia} \left(\sum \limits_{i=1}^m \alpha_i
^q\right)^{\frac{1}{q}}\left(\sum \limits_{i=1}^m \alpha_i
^p\right)^{-\frac{1}{p}}\le a.
\end{align}
Indeed, consider the problem
$$
\sum \limits _{i=1}^m |\alpha_i|^p \rightarrow \min, \quad \sum
\limits _{i=1}^m |\alpha_i|^q=\beta^q, \quad |\alpha_i|^q\le \beta
^q(1-\varepsilon_0), \quad 1\le i\le m.
$$
The compactness argument yields the existence of the point of
minimum, which will be denoted by $(\hat \alpha_1, \, \dots, \,
\hat \alpha_m)$. If $|\hat \alpha_i|^q<\beta^q(1-\varepsilon_0)$
for any $i=\overline{1, \, m}$, then by Lagrange's principle we
get $|\hat \alpha_i|=\beta k^{-\frac 1q}$, $i\in I$, $\hat
\alpha_i=0$, $i\notin I$, for some $I\subset \{1, \, \dots, \,
m\}$, ${\rm card}\, I=k$, $k\ge 2$. Thus,
$$
\left(\sum \limits _{i=1}^m |\hat \alpha_i|^q\right)^{\frac 1q}
\left(\sum \limits _{i=1}^m |\hat \alpha_i|^p\right)^{-\frac
1p}=k^{\frac 1q-\frac 1p}\le 2^{\frac 1q-\frac 1p}.
$$
Let $|\hat \alpha_{i_*}|^q=\beta^q (1-\varepsilon_0)$ for some
$i_*\in \{1, \, \dots, \, m\}$. Then
$$
\sum \limits _{i\ne i_*} |\hat\alpha_i|^p +|\hat\alpha_{i_*}|^p\ge
\left(\sum \limits_{i\ne i_*} |\hat\alpha_i|^q\right)^{\frac
pq}+|\hat\alpha_{i_*}|^p=\beta^p \varepsilon_0^{\frac pq}
+\beta^p(1-\varepsilon_0)^{\frac pq},
$$
$$
\left(\sum \limits _{i=1}^m |\hat \alpha_i|^q\right)^{\frac 1q}
\left(\sum \limits _{i=1}^m |\hat \alpha_i|^p\right)^{-\frac 1p}
\le \left(\varepsilon_0^{\frac pq}+(1-\varepsilon _0) ^{\frac
pq}\right)^{-\frac 1p} \stackrel{(\ref{1e})}{<} 1.
$$

The inequality (\ref{aia}) is proved. There exist
$\sigma_3=\sigma_3(p, \, q)\in (0, \, \sigma_2)$ and $\tilde
a=\tilde a(p, \, q)<1$ such that for any $\sigma\in (0, \, \sigma_3)$
$$
\left(\sum \limits_{i=1}^m \alpha_i
^q\right)^{\frac{p'}{q}}\left(\sum \limits _{i\in I_1} \alpha_i^p
+\sum \limits_{i\in I_2}\alpha_i^p \left(1-\frac{2p}{p'}
\sigma^{\frac 13}\right)\right)^{-\frac{p'}{p}}\le \tilde a^{p'}.
$$
Therefore,
$$
\tilde S\le \beta^{p'} +c^{p'}\tilde a^{p'} f(t_0) \stackrel
{(\ref{int}), (\ref{beta_le_1})}{\le} 1+c^{p'} \tilde
a^{p'}\frac{1}{1-\sigma^{\frac 13}}.
$$
Further, there are $\sigma_4=\sigma_4(p, \, q)\in (0, \,
\sigma_3)$ and $a_*=a_*(p, \, q)<1$ such that for $\sigma\in (0, \,
\sigma_4)$ the inequality $\tilde a^{p'}\frac{1}{1-\sigma^{\frac
13}}\le a_*^{p'}$ holds. Hence, there exists $c_0(p, \, q)\ge 2$
such that $\tilde S\le 1+c^{p'}a_*^{p'}\le
c^{p'}$ for $c\ge c_0(p, \, q)$.
\end{proof}

\begin{Cor}
\label{cor_of_main_lemma} Let $u$, $w:{\bf V}({\cal A})
\rightarrow [0, \, \infty)$. Suppose that $1<p<q<\infty$
and (\ref{w_lq_w_lq}) holds with $0<\sigma< \sigma_*(p, \, q)$.
Then
$$
\mathfrak{S}^{p,q}_{{\cal A},u,w} \underset{p,q}{\lesssim}
\sup _{\xi \in {\bf V}({\cal A})} u(\xi) \|w\|_{l_q({\cal A}_\xi)}.
$$
\end{Cor}
\begin{proof}
It suffices to consider $u$, $w:{\bf V}({\cal A})
\rightarrow (0, \, \infty)$. In this case, the assertion
follows from Lemmas \ref{hardy_cr} and \ref{bd_wq_est}.
\end{proof}

The following lemma gives a lower estimate.
\begin{Lem}
\label{low_est} Let $1<p\le q<\infty$, $\xi_*\in {\bf V}({\cal
A})$. Then
$$
\mathfrak{S}^{p,q}_{{\cal A}_{\xi_*},u,w} \underset{p,q}{\gtrsim}
\sup _{\xi \in {\bf V}({\cal A}_{\xi_*})} \left(\sum \limits
_{\xi_*\le\xi'\le \xi} u^{p'}(\xi')\right)^{\frac{1}{p'}}
\|w\|_{l_q({\cal A}_\xi)}.
$$
\end{Lem}
\begin{proof}
Let $\xi \in {\bf V}({\cal A}_{\xi_*})$. Define the tree ${\cal
D}$ by ${\bf V}({\cal D}) =(\bf V({\cal A}_{\xi_*})\backslash {\bf
V}({\cal A}_\xi)) \cup \{\xi\}$ and set $\Gamma=\{\xi\}$. Then
$({\cal D}, \, \Gamma)\in {\cal J}'_{\xi_*}$,  ${\bf V}({\cal
D}_\Gamma) ={\bf V}({\cal A}_{\xi_*}) \backslash {\bf V} ({\cal
A}_\xi)$, ${\cal A}_{\xi_*}\backslash {\cal D}_{\Gamma}={\cal
A}_{\xi}$. By Lemma \ref{hardy_cr},
\begin{align}
\label{spq_low_ex1_a} \mathfrak{S}^{p,q}_{{\cal A}_{\xi_*},u,w}
\underset{p,q} {\gtrsim} \beta_{{\cal
D},\Gamma}^{-1}\|w\|_{l_q({\cal A}_\xi)}.
\end{align}
We have
$$
\beta_{{\cal D},\Gamma}=\inf \left\{\|f\|_{l_p({\cal D})} :\; \sum
\limits _{\xi_*\le \xi'\le \xi} u(\xi')|f(\xi')| =1\right\}=
$$
$$
=\inf \left\{\left(\sum \limits _{\xi_*\le \xi'\le \xi}|f(\xi')|
^p\right)^{1/p}:\; \sum \limits _{\xi_*\le \xi'\le \xi}
u(\xi')|f(\xi')|=1\right\}=\left(\sum \limits _{\xi_*\le \xi'\le
\xi}u^{p'}(\xi')\right)^{-1/p'}.
$$
This completes the proof.
\end{proof}

Let $({\cal A}, \, \xi_0)$ be a tree, $u$, $w:{\bf V}({\cal A})
\rightarrow [0, \, \infty)$, $\xi_*\in {\bf V}_{j_0}^{\cal
A}(\xi_0)$, $m\in \Z_+\cup \{+\infty\}$,
$j_0<j_1<j_2<\dots<j_k<\dots$, $J=\{j_k\}_{0\le k<m+1}$. For $0\le
k<m+1$ denote by ${\cal G}_k$ the maximal subgraph of ${\cal A}$
on the set of vertices $\cup _{j_k\le j<j_{k+1}} {\bf
V}_{j-j_0}^{\cal A}(\xi_*)$, and by $\{{\cal A}_{k,i}\}_{i\in
I_k}$, the set of its connected components. Let $\xi_{k,i}$ be the
minimal vertex of the tree ${\cal A}_{k,i}$.

Define the tree ${\cal A}_J$ by
\begin{align}
\label{bf_v_a_j_bf_v1} {\bf V}({\cal A}_J)=\{\xi_{k,i}\}_{0\le
k<m+1, \, i\in I_k}, \quad {\bf V}_1^{{\cal A}_J}(\xi_{k,i})={\bf
V}^{\cal A}_{j_{k+1} -j_k}(\xi_{k,i}), \quad 0\le k< m.
\end{align}
For $0\le k<m+1$, $i\in I_k$ we set
\begin{align}
\label{u_w} u_J(\xi_{k,i})=\|u\|_{l_{p'}({\cal A}_{k,i})}, \quad
w_J(\xi_{k,i}) = \|w\|_{l_q({\cal A}_{k,i})}.
\end{align}
\begin{Lem}
\label{red} The inequality $\mathfrak{S}^{p,q}_{{\cal
A}_{\xi_*},u,w}\le \mathfrak{S}^{p,q}_{{\cal A}_J,u_J,w_J}$ holds.
\end{Lem}
\begin{proof}
Let $f:{\bf V}({\cal A}_{\xi_*}) \rightarrow \R_+$,
$\|f\|_{l_p({\cal A}_{\xi_*})}=1$. Denote
$f_J(\xi_{k,i}) =\|f\| _{l_p({\cal A}_{k,i})}$, $0\le k\le m$,
$i\in I_k$. Then $\|f_J\| _{l_p({\cal A}_J)}=1$.

Let $\xi\in {\bf V}({\cal A}_{k,i})$. Then for any $0\le
l\le k$ there exists $i_l\in I_l$ such that $\xi_{l,i_l}\le\xi$.
This together the H\"{o}lder's inequality yields
$$
\sum \limits _{\xi_*\le \xi'\le \xi} u(\xi')f(\xi') \le \sum
\limits _{l=0}^k \sum \limits _{\xi'\in {\bf V}({\cal A}_{l,i_l})}
u(\xi')f(\xi') \stackrel{(\ref{u_w})}{\le} \sum \limits _{l=0}^k
u_J(\xi_{l,i_l})f_J(\xi_{l,i_l}) =
$$
$$
=\sum \limits _{\zeta'\in {\bf V}({\cal A}_J),
\, \zeta'\le \xi_{k,i}} u_J(\zeta')f_J(\zeta').
$$
Hence,
$$
\sum \limits _{\xi\in {\bf V}({\cal A}_{\xi_*})} w^q(\xi) \left(
\sum \limits _{\xi_*\le \xi' \le \xi} u(\xi') f(\xi')\right)^q =
\sum \limits _{k=0}^m \sum \limits _{i\in I_k} \sum \limits _{\xi
\in {\bf V}({\cal A}_{k,i})}w^q(\xi) \left( \sum \limits
_{\xi_*\le \xi' \le \xi} u(\xi') f(\xi')\right)^q\le
$$
$$
\le \sum \limits _{k=0}^m \sum \limits _{i\in I_k} \sum \limits
_{\xi \in {\bf V}({\cal A}_{k,i})}w^q(\xi)\left(\sum \limits
_{\zeta'\in {\bf V}({\cal A}_J), \, \zeta'\le \xi_{k,i}}
u_J(\zeta')f_J(\zeta')\right)^q \stackrel{(\ref{u_w})}{=}
$$
$$
=\sum \limits _{k=0}^m \sum
\limits _{i\in I_k} w^q_J(\xi_{k,i})\left(\sum \limits _{\zeta'\in
{\bf V}({\cal A}_J), \, \zeta'\le \xi_{k,i}}
u_J(\zeta')f_J(\zeta')\right)^q = $$$$=\sum \limits _{\zeta \in
{\bf V}({\cal A}_J)} w_J^q(\zeta)\left(\sum \limits _{\zeta'\in
{\bf V}({\cal A}_J), \, \zeta'\le \zeta}
u_J(\zeta')f_J(\zeta')\right)^q\le \left[
\mathfrak{S}^{p,q}_{{\cal A}_J,u_J,w_J}\right]^q.
$$
This completes the proof.
\end{proof}

\renewcommand{\proofname}{\bf Proof of Theorem \ref{main_plq}}
\begin{proof}
Denote by $\xi_0$ the minimal vertex of ${\cal A}$ and
set $\hat{\mathfrak{Z}}=(K, \, \lambda, \, l_0, \, p, \, q)$.

Let $\sigma_*=\sigma_*(p, \, q)\in (0, \, 1)$ be such as in Lemma
\ref{bd_wq_est}, and let $t_*=t_*(\hat{\mathfrak{Z}})\in \N$
be such that $\lambda^{t_*}\le\frac{\sigma_*}{2}$. Set
$l_*=l_0t_*$. For $m\in \N$ we define the function
$u_m:{\bf V}({\cal A}) \rightarrow \R_+$ by
$$
u_m(\xi) =\left\{ \begin{array}{c} u(\xi), \quad \xi \in {\bf
V}_j^{\cal A}(\xi_0), \quad j\le l_*m, \\ 0, \quad \xi \in {\bf
V}_j^{\cal A}(\xi_0), \quad j> l_*m. \end{array}\right.
$$
Prove that
\begin{align}
\label{spqm} \mathfrak{S}^{p,q}_{{\cal A},u_m,w}
\underset{\hat{\mathfrak{Z}}}{\lesssim} \sup _{\xi\in {\bf
V}({\cal A})} u(\xi) \|w\|_{l_q({\cal A}_\xi)}.
\end{align}
This together with B. Levi's theorem gives the desired estimate.

For $k\in \Z_+$ we set $j_k=l_*k$. Denote $J=\{j_k\} _{0\le k \le
m}$ and and define the tree ${\cal A}_J$ by
(\ref{bf_v_a_j_bf_v1}). Then Lemma \ref{red} yields
\begin{align}
\label{spqm1} \mathfrak{S}^{p,q}_{{\cal A},u_m,w}\le
\mathfrak{S}^{p,q}_{{\cal A}_J,(u_m)_J,w_J};
\end{align}
here $(u_m)_J$, $w_J$ are defined by (\ref{u_w}).

Let $0\le k\le m-1$, $\xi_{k,i}\in {\bf V}({\cal A}_J)$,
$\xi_{k+1,i'}\in {\bf V}_1^{{\cal A}_J}(\xi_{k,i})$,
$\xi_{k,i}=\eta_0<\eta_1<\dots<\eta_{t_*}=\xi_{k+1,i'}$, $\eta_j
\in {\bf V}^{\cal A}_{l_0}(\eta_{j-1})$, $1\le j\le t_*$.
Then
$$
\frac{\|w_J\|_{l_q(({\cal A}_J)_{\xi_{k+1,i'}})}}
{\|w_J\|_{l_q(({\cal A}_J)_{\xi_{k,i}})}}=\frac{\|w\|_{l_q({\cal
A}_{\xi_{k+1,i'}})}}{\|w\|_{l_q({\cal A}_{\xi_{k,i}})}} =\prod
_{j=1}^{t_*} \frac{\|w\|_{l_q({\cal
A}_{\eta_j})}}{\|w\|_{l_q({\cal A} _{\eta_{j-1}})}}
\stackrel{(\ref{wq_lam})}{\le} \lambda^{t_*} \le
\frac{\sigma_*}{2}.
$$
By Corollary \ref{cor_of_main_lemma},
\begin{align}
\label{sujwj} \mathfrak{S}^{p,q}_{{\cal A}_J,(u_m)_J,w_J}
\underset{p,q}{\lesssim} \sup _{\zeta\in {\bf V}({\cal A}_J)}
(u_m)_J(\zeta) \|w\|_{l_q(({\cal A}_J)_\zeta)} =\sup _{0\le k\le
m, \; i\in I_k} \|u_m\|_{l_{p'}({\cal A}_{k,i})} \|w\| _{l_q({\cal
A}_{\xi_{k,i}})}.
\end{align}
If $k<m$, then by (\ref{cv1k}) we have ${\rm card}\, {\bf V}({\cal
A}_{k,i}) \underset{\hat{\mathfrak{Z}}}{\lesssim} 1$; this
together the first relation in (\ref{wq_lam}) yield that
$\|u_m\|_{l_{p'}({\cal A}_{k,i})} \underset{\hat{\mathfrak{Z}}}
{\lesssim} u(\xi_{k,i})$. If $k=m$, then $\|u_m\|_{l_{p'}({\cal
A}_{k,i})}= u(\xi_{k,i})$. This together with (\ref{spqm1}) and
(\ref{sujwj}) implies (\ref{spqm}).

The lower estimate follows from Lemma \ref{low_est}.
\end{proof}
\renewcommand{\proofname}{\bf Proof}

Consider two examples.

{\bf Example 1.} Suppose that there is $C_*\ge 1$ such that for any
$j\in \Z_+$, $j'\ge j$, $\xi \in {\bf V}_j^{\cal A}(\xi)$
\begin{align}
\label{card_vjj0}  C_*^{-1}\cdot 2^{\psi(j')-\psi(j)} \le {\rm
card} \, {\bf V}^{\cal A}_{j'-j}(\xi) \le C_*\cdot
2^{\psi(j')-\psi(j)}, \quad 2^{\psi(t)} =2^{\theta
st}\Lambda_*(2^{st});
\end{align}
here $\theta>0$, $s\in \N$, $\Lambda_*: (0, \, \infty) \rightarrow
(0, \, \infty)$ is an absolutely continuous function such that $\lim
_{y\to \infty} \frac{y\Lambda_*'(y)}{\Lambda_*(y)}=0$. Suppose that for
$\xi \in {\bf V}_j^{\cal A}(\xi_0)$
\begin{align}
\label{ujwj} u(\xi)=u_j=2^{\frac{\theta sj}{q}} \Psi_u(2^{sj}),
\quad w(\xi) = w_j= 2^{-\frac{\theta sj}{q}} \Psi_w(2^{sj}).
\end{align}
Here $\Psi_u$, $\Psi_w: (0, \, \infty) \rightarrow (0, \, \infty)$
are absolutely continuous functions such that $\lim _{y\to \infty}
\frac{y\Psi_u'(y)}{\Psi_u(y)}=\lim _{y\to \infty}
\frac{y\Psi_w'(y)}{\Psi_w(y)}=0$.

Set ${\mathfrak{Z}} = (u, \, w, \, \psi, \, C_*, \, p,
\, q)$.

For $j_0\in \Z_+$ we write
$$
M_{j_0} =\sup _{j\in \Z_+, \, j\ge j_0} \Psi_u(2^{sj}) \left(\sum
\limits _{i\ge j} \Psi_w^q(2^{si}) \frac{\Lambda_*(2^{si})}
{\Lambda_*(2^{sj})}\right)^{\frac 1q}.
$$

The proof of the following lemma is straightforward
and will be omitted.
\begin{Lem}
\label{sum_lem} Let $\Lambda_*:(0, \, +\infty) \rightarrow
(0, \, +\infty)$ be an absolutely continuous function such that
$\lim \limits_{y\to +\infty}\frac{y\Lambda _*'(y)}
{\Lambda _*(y)}=0$. Then for any
$\varepsilon >0$
\begin{align}
\label{rho_yy} t^{-\varepsilon}
\underset{\varepsilon,\Lambda_*}{\lesssim}
\frac{\Lambda_*(ty)}{\Lambda_*(y)}\underset{\varepsilon,
\Lambda_*}{\lesssim}
t^\varepsilon,\quad 1\le y<\infty, \;\; 1\le t<\infty.
\end{align}
\end{Lem}
\begin{Trm}
\label{ex1_plq} Let $M_{j_0}<\infty$, $\xi_*\in {\bf V}_{j_0}
^{\cal A} (\xi_0)$. Then $\mathfrak{S}^{p,q}_{{\cal
A}_{\xi_*},u,w} \underset{{\mathfrak{Z}}}{\asymp} M_{j_0}$.
\end{Trm}
\begin{proof}
Let $\xi \in {\bf V}_{j-j_0}^{\cal A}(\xi_*)$. Then
$$
\|w\| _{l_q({\cal A}_{\xi})}\stackrel{(\ref{card_vjj0})}
{\underset{{\mathfrak{Z}}}{\asymp}}\left(\sum \limits _{j'\ge j}
w^q(j')2^{\psi(j')-\psi(j)}\right)^{\frac 1q}
\stackrel{(\ref{card_vjj0}),(\ref{ujwj})}{=}
$$
$$
=\left(\sum \limits _{j'\ge j} 2^{-\theta sj'}
\Psi_w^q(2^{sj'}) \cdot 2^{\theta
s(j'-j)}\frac{\Lambda_*(2^{sj'})}
{\Lambda_*(2^{sj})}\right)^{\frac 1q}
=2^{-\frac{\theta sj}{q}} \left(\sum \limits _{j'\ge j}
\Psi_w^q(2^{sj'})\frac{\Lambda_*(2^{sj'})}
{\Lambda_*(2^{sj})}\right)^{\frac 1q},
$$
i.e.,
\begin{align}
\label{ststst}
\|w\| _{l_q({\cal A}_\xi)}\underset{{\mathfrak{Z}}}{\asymp}
2^{-\frac{\theta sj}{q}} \left(\sum \limits _{j'\ge j}
\Psi_w^q(2^{sj'})\frac{\Lambda_*(2^{sj'})}
{\Lambda_*(2^{sj})}\right)^{\frac 1q}.
\end{align}
For any $l_0\in \N$
$$
\sum \limits _{j'\ge j+l_0} \frac{\Lambda_*(2^{sj'})}
{\Lambda_*(2^{s(j+l_0)})}\Psi_w^q(2^{sj'}) \le
\frac{\Lambda_*(2^{sj})}{\Lambda_*(2^{s(j+l_0)})}\sum \limits
_{j'\ge j} \frac{\Lambda_*(2^{sj'})}
{\Lambda_*(2^{sj})}\Psi_w^q(2^{sj'}).
$$
Therefore, for any $\xi\in {\bf V}^{\cal A}_{j-j_0}(\xi_*)$,
$\xi'\in {\bf V}_{l_0}^{\cal A}(\xi)$
$$
\frac{\|w\|_{l_q({\cal A}_{\xi'})}} {\|w\|_{l_q({\cal A}_\xi)}}
\stackrel{(\ref{ststst})}{\underset{{\mathfrak{Z}}}{\lesssim}}
2^{-\frac{\theta sl_0}{q}} \frac{\Lambda_*^{\frac 1q}(2^{sj})}
{\Lambda_*^{\frac 1q}(2^{s(j+l_0)})} \stackrel{(\ref{rho_yy})}
{\underset{{\mathfrak{Z}}} {\lesssim}} 2^{-\frac{\theta
sl_0}{2q}}.
$$
Hence, for sufficiently large $l_0$ there is $\lambda\in (0, \,
1)$ such that $\frac{\|w\|_{l_q({\cal A}_{\xi'})}}{\|w\|_{l_q
({\cal A}_\xi)}} \le \lambda$, $\xi'\in {\bf V}_{l_0}^{\cal
A}(\xi)$.

For any $\xi\in {\bf V}_{j-j_0}^{\cal A}(\xi_*)$
$$
\|w\|_{l_q({\cal A}_\xi)}u(\xi) \stackrel{(\ref{ujwj}),
(\ref{ststst})}{\underset{{\mathfrak{Z}}}
{\asymp}} 2^{-\frac{\theta sj}{q}}
\left(\sum \limits _{j'\ge j}
\Psi_w^q(2^{sj'})\frac{\Lambda_*(2^{sj'})}
{\Lambda_*(2^{sj})}\right)^{\frac 1q}\cdot
2^{\frac{\theta sj}{q}} \Psi_u(2^{sj})=
$$
$$
= \left(\sum \limits _{j'\ge j}
\Psi_w^q(2^{sj'})\frac{\Lambda_*(2^{sj'})}
{\Lambda_*(2^{sj})}\right)^{\frac 1q}\Psi_u
(2^{sj}).
$$
It remains to take the supremum over $j\ge j_0$ and apply Theorem
\ref{main_plq}.
\end{proof}

{\bf Example 2.} Suppose that there exists $C_*\ge 1$ such that
for any $j\in \Z_+$, $j'\ge j$, $\xi \in {\bf V}_{j}^{\cal
A}(\xi)$
\begin{align}
\label{cs1jgt}
C_*^{-1}\cdot 2^{\psi(j')-\psi(j)}\le {\rm
card} \, {\bf V}^{\cal A}_{j'-j}(\xi) \le C_*\cdot
2^{\psi(j')-\psi(j)}, \quad 2^{\psi(j)} =j^{\gamma_*}\tau_*(j);
\end{align}
here $\gamma_*>0$, $\tau_*: (0, \, \infty) \rightarrow (0, \,
\infty)$ is an absolutely continuous function such that $\lim _{y\to
\infty} \frac{y\tau_*'(y)}{\tau_*(y)}=0$. Suppose that for any $\xi \in {\bf
V}^{\cal A}_j(\xi_0)$
\begin{align}
\label{uujaurhou} u(\xi)=u_j=j^{-\alpha_u} \rho_u(j), \quad w(\xi)
= w_j= j ^{-\alpha_w} \rho_w(j),
\end{align}
where $\rho_u$, $\rho_w: (0, \, \infty) \rightarrow (0, \,
\infty)$ are absolutely continuous functions such that $\lim
_{y\to \infty} \frac{y\rho_u'(y)}{\rho_u(y)}=\lim _{y\to \infty}
\frac{y\rho_w'(y)}{\rho_w(y)}=0$.

As in Example 1, we set ${\mathfrak{Z}} =
(u, \, w, \, \psi, \, C_*, \, p, \, q)$.

\begin{Trm}
Suppose that $j_0=2^{k_0}$, $k_0\in \Z_+$, $\xi_*\in {\bf V}
_{j_0} ^{\cal A}(\xi_0)$.
\begin{enumerate}
\item Let $-\alpha_w+ \frac 1q+\frac{\gamma_*}{q} <0$. Set $\alpha=\alpha_u+\alpha_w$,
$\rho(t)=\rho_u(t)\rho_w(t)$. If $M_{j_0}:=\sup _{j\ge j_0}
j^{-\alpha+\frac 1q+\frac{1}{p'}}\rho(j)<\infty$, then
$\mathfrak{S}_{{\cal A}_{\xi_*},u,w}^{p,q}
\underset{\hat{\mathfrak{Z}}}{\asymp} M_{j_0}$. In particular, if
$-\alpha+\frac 1q+\frac{1}{p'}<0$, then $\mathfrak{S}_{{\cal
A}_{\xi_*},u,w}^{p,q} \underset{\hat{\mathfrak{Z}}}{\asymp}
j_0^{-\alpha+\frac 1q+\frac{1}{p'}}\rho(j_0)$.
\item Let $-\alpha_w+ \frac 1q+\frac{\gamma_*}{q}=0$,
$-\alpha_u+\frac{1}{p'}-\frac{\gamma_*}{q}=0$,
$$
\tilde M_{k_0}:=\sup _{k\in \Z_+} \rho_u(2^{k_0+k}) \left( \sum
\limits _{t\ge k} \rho_w^q(2^{k_0+t}) \frac{\tau_*(2^{k_0+t})}
{\tau_*(2^{k_0+k})}\right)^{\frac 1q}<\infty.
$$
Then $\mathfrak{S}_{{\cal A}_{\xi_*},u,w}^{p,q}
\underset{\hat{\mathfrak{Z}}}{\asymp} \tilde M_{k_0}$.
\end{enumerate}
\end{Trm}
\begin{proof}
Prove the upper estimate. Let $j_k=2^{k_0+k}$, $k\in \Z_+$,
$J=\{j_k\}_{k\in \Z_+}$. Define the tree ${\cal A}_J$ and weights
$w_J$, $u_J$ by (\ref{bf_v_a_j_bf_v1}), (\ref{u_w}).
Since
\begin{align}
\label{cjjkaxiki1} {\rm card}\, \{\xi\in {\bf V}_{j-j_k}^{\cal A}
(\xi_{k,i})\}\stackrel{(\ref{rho_yy}),(\ref{cs1jgt})}
{\underset{\mathfrak{Z}}{\lesssim}} 1, \quad j_k\le j<j_{k+1},
\end{align}
we have ${\rm card}\, {\bf V}({\cal A}_{k,i}) \underset{
{\mathfrak{Z}}}{\asymp} 2^{k_0+k}$. Hence,
$$
(u_J)({\xi_{k,i}}) \stackrel{(\ref{uujaurhou})}
{\underset{{\mathfrak{Z}}}{\asymp}}
2^{\left(-\alpha_u+\frac{1}{p'}\right)(k_0+k)}\rho_u(2^{k_0+k}),
\quad (w_J)({\xi_{k,i}}) \stackrel{(\ref{uujaurhou})}
{\underset{{\mathfrak{Z}}}{\asymp}}
2^{\left(-\alpha_w+\frac 1q\right)(k_0+k)}\rho_w(2^{k_0+k}),
$$
$$
{\rm card}\, {\bf V}_{k'-k}^{{\cal A}_J}(\xi_{k,i})
={\rm card}\, {\bf V}_{j_{k'}-j_k}^{{\cal A}}(\xi_{k,i})
\stackrel{(\ref{cs1jgt})}{\underset{{\mathfrak{Z}}}{\asymp}}
2^{\psi_J(k')-\psi_J(k)},
\quad k'\ge k, \quad 2^{\psi_J(l)}=2^{\gamma_*(k_0+l)}
\tau_*(2^{k_0+l}).
$$

In the case 1 we get
$$
\|w_J\|_{l_q(({\cal A}_J)_{\xi_{k,i}})} \underset
{{\mathfrak{Z}}} {\asymp} 2^{\left(-\alpha_w+\frac
1q\right)(k_0+k)}\rho_w(2^{k_0+k}),
$$
\begin{align}
\label{wj} \sup _{k\in \Z_+}\|w_J\|_{l_q(({\cal
A}_J)_{\xi_{k,i}})} u_J(\xi_{k,i}) \underset {{\mathfrak{Z}}}
{\asymp} \sup _{l\ge k_0} 2^{\left(-\alpha+\frac
1q+\frac{1}{p'}\right)l} \rho(2^l)
\underset{\hat{\mathfrak{Z}}}{=} M_{j_0}.
\end{align}
In the case 2 we have
\begin{align}
\label{uj} (u_J)({\xi_{k,i}}) \underset{{\mathfrak{Z}}}{\asymp}
2^{\frac{\gamma_*(k_0+k)}{q}} \rho_u(2^{k_0+k}), \quad
(w_J)({\xi_{k,i}}) \underset{{\mathfrak{Z}}}{\asymp}
2^{-\frac{\gamma_*(k_0+k)}{q}} \rho_w(2^{k_0+k}).
\end{align}
The further arguments are the same as in Example 1.

In order to prove the lower estimate, we notice that $\|w\|_{l_q
({\cal A}_{\xi_{k,i}})}=\|w_J\|_{l_q(({\cal A}_J)_{\xi_{k,i}})}$,
$\left(\sum \limits _{\xi_*\le \xi'\le \xi_{k,i}} u^{p'}(\xi')
\right)^{\frac{1}{p'}} \stackrel{(\ref{cjjkaxiki1})}
{\underset{\mathfrak{Z}}{\gtrsim}} (u_J)({\xi_{k,i}})$ for $k\ge
1$ and apply Lemma \ref{low_est} together with (\ref{wj}) and
(\ref{uj}).
\end{proof}

\section{An estimate for the norm of a weighted summation operator
on a tree: case $p\ge q$}

Suppose that conditions of Theorem \ref{p_ge_q} hold.

We shall use the following notation.

Let $k\in \N\cup \{\infty\}$, ${\cal T}$, ${\cal T}_1, \, \dots ,
\, {\cal T}_k$ be trees that have no common vertices, let $v_1, \,
\dots , \, v_k\in {\bf V}({\cal T})$, $w_j\in {\bf V}({\cal
T}_j)$, $j=1, \, \dots , \, k$. Denote by
$$
{\bf J}({\cal T}, \, {\cal T}_1, \, \dots , \, {\cal T}_k; v_1, \,
w_1, \, \dots , \, v_k, \, w_k)
$$
the tree obtained from ${\cal T}$, ${\cal T}_1, \, \dots , \,
{\cal T}_k$ by connecting the vertices $v_j$ and $w_j$ by an edge
for each $j=1, \, \dots , \, k$.

Let $({\cal D}, \, \xi_0)$ be a tree, $\xi \in {\bf V}({\cal
D})$, $n\in \N$, let $T=\{A_1, \, \dots, \, A_n\}$ be a partition of ${\bf
V}^{\cal D}_1(\xi)$ into nonempty subsets,
$A_j=\{\xi_{j,i}\}_{i=1}^{k_j}$. Define the graph $G_{\xi,T}({\cal
D})$ as follows.
\begin{enumerate}
\item Let $\xi=\xi_0$. Then we denote by $G_{\xi,T}({\cal D})$ the
graph that is a disjoint union of trees $\tilde {\cal
D}_j:={\bf J}(\{\eta_j\}, \, {\cal D}_{\xi_{j,1}}, \, \dots, {\cal
D}_{\xi_{j,k_j}}; \, \eta_j, \, \xi_{j,1}, \, \dots, \, \eta_j, \,
\xi_{j,k_j})$.
\item Let $\xi>\xi_0$, and let $\eta$ be the direct predecessor of $\xi$.
Then we set
$$
G_{\xi,T}({\cal D})={\bf J}({\cal D}\backslash {\cal D}_\xi, \,
\tilde {\cal D}_1, \, \dots, \, \tilde {\cal D}_n; \, \eta, \,
\eta_1, \, \dots, \, \eta, \, \eta_n),
$$
where the vertices $\eta_j$ and trees $\tilde {\cal D}_j$ are
defined above.
\end{enumerate}
Let $\overline{u}, \, \overline{w}:{\bf V}({\cal D}) \rightarrow
(0, \, \infty)$, $\xi\in {\bf V}({\cal D})$.
Define weights $\overline{u}_{\xi,T}$ and $\overline{w}_{\xi,T}$ on
the graph $G_{\xi,T}({\cal D})$ as follows. If $\zeta\in {\bf
V}({\cal D})\backslash {\bf V}({\cal D}_\xi)$ or $\zeta \in \cup
_{j=1}^n \cup _{i=1}^{k_j} {\bf V}({\cal D}_{\xi_{j,i}})$, then
we set $\overline{u}_{\xi,T}(\zeta)=\overline{u}(\zeta)$,
$\overline{w}_{\xi,T}(\zeta)=\overline{w}(\zeta)$; if $\zeta=\eta_j$ for some
$j\in \{1, \, \dots, \, n\}$, then we set
\begin{align}
\label{uxit_wxit_ch}
\overline{u}_{\xi,T}(\eta_j)=n^{\frac 1p} \overline{u}(\xi), \quad
\overline{w}_{\xi,T}(\eta_j)=n^{-\frac 1q} \overline{w}(\xi).
\end{align}

If each element of $T$ is a singlepoint,
then we denote
\begin{align}
\label{u_xi_t_g_xi_t} G_{\xi,T}({\cal D})=G_\xi({\cal D}), \quad
\overline{u}_{\xi,T}=\overline{u}_{\xi},\quad
\overline{w}_{\xi,T}=\overline{w}_{\xi}.
\end{align}

\begin{Lem} \label{razv}
For any $1\le p, \, q\le \infty$
\begin{align}
\label{spq} \mathfrak{S}^{p,q}_{\overline{u}, \overline{w}, {\cal
D}}\le \mathfrak{S}^{p,q}_{\overline{u}_{\xi,T},
\overline{w}_{\xi,T}, G_{\xi,T}({\cal D})}.
\end{align}
\end{Lem}
\begin{proof}
Let $f:{\bf V}({\cal D})\rightarrow \R_+$, $\|f\|_{l_p({\cal
D})}=1$. Define the function $f_{\xi,T}: {\bf V}(G_\xi({\cal D}))
\rightarrow \R_+$ as follows: we set $f_{\xi,T}(\zeta)=f(\zeta)$ for
$\zeta\in {\bf V}({\cal D}) \backslash {\bf V}({\cal D}_\xi)$ or for
$\zeta \in \cup _{j=1}^n \cup _{i=1}^{k_j} {\bf V}({\cal
D}_{\xi_{j,i}})$, and we set $f_{\xi,T}(\eta_j)=n^{-\frac 1p} f(\xi)$, $1\le
j\le n$. Then $\|f_{\xi,T}\|_{l_p(G_{\xi,T}({\cal
D}))}=\|f\|_{l_p({\cal D})}$. We have
$$
\sum \limits _{\zeta \in {\bf V}({\cal D})} \overline{w}^q(\zeta)
\left(\sum \limits _{\zeta'\le \zeta}
\overline{u}(\zeta')f(\zeta')\right)^q =\sum \limits _{\zeta \in
{\bf V}({\cal D}) \backslash {\bf V}({\cal
D}_\xi)}\overline{w}^q(\zeta) \left(\sum \limits _{\zeta'\le
\zeta} \overline{u}(\zeta')f(\zeta')\right)^q+
$$
$$
+\overline{w}^q(\xi)\left(\sum \limits _{\zeta'\le \xi}
\overline{u}(\zeta')f(\zeta')\right)^q+ \sum \limits _{j=1}^n \sum
\limits _{i=1}^{k_j} \sum \limits _{\zeta \in {\bf V}({\cal
D}_{\xi_{j,i}})} \overline{w}^q (\zeta) \left(\sum \limits
_{\zeta'\le \zeta} \overline{u} (\zeta')f(\zeta')\right)^q=:S.
$$
Since ${\bf V}({\cal D}) \backslash {\bf V}({\cal D}_\xi)
\subset {\bf V}(G_{\xi,T}({\cal D}))$, by definitions of
$\overline{u}_{\xi,T}$, $\overline{w}_{\xi,T}$ and ${f}_{\xi,T}$
we get
\begin{align}
\label{s1} \sum \limits _{\zeta \in {\bf V}({\cal D}) \backslash
{\bf V}({\cal D}_\xi)}\overline{w}^q(\zeta) \left(\sum \limits
_{\zeta'\le \zeta} \overline{u}(\zeta')f(\zeta')\right)^q = \sum
\limits _{\zeta \in {\bf V}({\cal D}) \backslash {\bf V}({\cal
D}_\xi)}\overline{w}_{\xi,T}^q(\zeta) \left(\sum \limits
_{\zeta'\le \zeta} \overline{u}_{\xi,T}(\zeta') f_{\xi,T}
(\zeta')\right)^q.
\end{align}
Let $1\le j\le n$. Then
$$
\overline{w}_{\xi,T}^q(\eta_j) \left(\sum \limits _{\zeta' \in
{\bf V}(G_{\xi,T}({\cal D})), \, \zeta'\le \eta_j}
\overline{u}_{\xi,T}(\zeta') {f}_{\xi,T}(\zeta')\right)^q=
$$
$$
=n^{-1} \overline{w}^q(\xi)\left(\sum \limits _{\zeta' \in {\bf
V}({\cal D}), \, \zeta'< \xi} \overline{u}(\zeta')
{f}(\zeta')+n^{\frac 1p}\overline{u}(\xi) \cdot n^{-\frac 1p}
f(\xi)\right)^q=
$$
$$
=n^{-1} \overline{w}^q(\xi)\left(\sum \limits _{\zeta' \in {\bf
V}({\cal D}), \, \zeta'\le \xi} \overline{u}(\zeta')
{f}(\zeta')\right)^q.
$$
Hence,
\begin{align}
\label{s2} \overline{w}^q(\xi)\left(\sum \limits _{\zeta' \in {\bf
V}({\cal D}), \, \zeta'\le \xi} \overline{u}(\zeta')
{f}(\zeta')\right)^q=\sum \limits _{j=1}^n
\overline{w}_{\xi,T}^q(\eta_j) \left(\sum \limits _{\zeta' \in
{\bf V}(G_{\xi,T}({\cal D})), \, \zeta'\le \eta_j}
\overline{u}_{\xi,T}(\zeta') {f}_{\xi,T}(\zeta')\right)^q.
\end{align}
Let $\zeta \in {\bf V}({\cal D}_{\xi_{j,i}})$, $1\le j\le n$,
$1\le i\le k_j$. Then
$$
\sum \limits _{\zeta'\in {\bf V}(G_{\xi,T}({\cal D})), \,
\zeta'\le \zeta} \overline{u}_{\xi,T}(\zeta')
f_{\xi,T}(\zeta')=\sum \limits _{\zeta'\in {\bf V}(G_{\xi,T}({\cal
D})), \, \zeta'\le \zeta, \, \zeta'\ne \eta_j}
\overline{u}_{\xi,T}(\zeta') f_{\xi,T}(\zeta')+
\overline{u}_{\xi,T}(\eta_j)f_{\xi,T}(\eta_j)=
$$
$$
=\sum \limits _{\zeta'\in {\bf V}({\cal D}), \, \zeta'\le \zeta,
\, \zeta'\ne \xi} \overline{u}(\zeta') f(\zeta')+ n^{\frac
1p}\overline{u}(\xi)\cdot n^{-\frac 1p}f(\xi)= \sum \limits
_{\zeta'\in {\bf V}({\cal D}), \, \zeta'\le \zeta}
\overline{u}(\zeta') f(\zeta').
$$
Therefore,
\begin{align}
\label{s3} \overline{w}^q(\zeta)\left(\sum \limits _{\zeta'\in {\bf
V}({\cal D}), \, \zeta'\le \zeta} \overline{u}(\zeta')
f(\zeta')\right)^q = \overline{w}_{\xi,T}^q(\zeta)\left(\sum
\limits _{\zeta'\in {\bf V}(G_{\xi,T}({\cal D})), \, \zeta'\le
\zeta} \overline{u}_{\xi,T}(\zeta') f_{\xi,T}(\zeta')\right)^q.
\end{align}

From (\ref{s1}), (\ref{s2}) and (\ref{s3}) it follows that
$$
S= \sum \limits _{\zeta \in {\bf V}(G_{\xi,T}({\cal D}))}
\overline{w}_{\xi,T}^q (\zeta) \left(\sum \limits _{\zeta'\le
\zeta} \overline{u}_{\xi,T}(\zeta')f_{\xi,T}(\zeta')\right)^q\le
\left(\mathfrak{S}^{p,q}_{\overline{u}_{\xi,T},\overline{w}_{\xi,T},
G_{\xi,T}({\cal D})}\right)^q.
$$
This completes the proof of (\ref{spq}).
\end{proof}

Denote by $[{\cal A}]_{\le n}$ a subtree in ${\cal A}$ such that
$$
{\bf V}([{\cal A}]_{\le n})= \cup _{j=0}^n {\bf V}_j^{\cal
A}(\xi_0).
$$

\renewcommand{\proofname}{\bf Proof of Theorem \ref{p_ge_q}}
\begin{proof}
It suffices to consider the case $p<\infty$ and $N<\infty$.

For $0\le j\le N$ we construct the graph ${\cal G}_{j,{\cal A}}$
and the functions $u^{(j)}$, $w^{(j)}:{\bf V}({\cal G}_{j,{\cal
A}}) \rightarrow (0, \, \infty)$ with the following properties:
\begin{enumerate}
\item ${\cal G}_{N,{\cal A}}={\cal A}$, $u^{(N)}=u$, $w^{(N)}=w$.
\item If $1\le j\le N-1$, then ${\cal G}_{j,{\cal A}}$ is a tree with
the minimal vertex $\xi_0$; here
\begin{align}
\label{gjaj} [{\cal G}_{j,{\cal A}}]_{\le j-1}=[{\cal A}]_{\le
j-1}, \quad {\bf V}_{N}^{{\cal G}_{j,{\cal A}}}(\xi_0)={\bf
V}_{\max}({\cal G}_{j,{\cal A}});
\end{align}
\begin{align}
\label{cv1gj} {\rm card}\, {\bf V}_1^{{\cal G}_{j,{\cal A}}}(\xi)
={\rm card}\, {\bf V}_{N-j+1}^{\cal A}(\xi), \quad
\text{if}\quad \xi\in {\bf V}_{j-1}^{{\cal G}_{j,{\cal
A}}}(\xi_0);
\end{align}
\begin{align}
\label{cvvv} {\rm card}\, {\bf V}_1^{{\cal G}_{j,{\cal A}}}(\xi)
=1, \quad \text{if} \quad \xi\in {\bf V}_i^{{\cal G}_{j,{\cal
A}}}(\xi_0), \quad j\le i\le N-1.
\end{align}
In addition,
\begin{align}
\label{uj_uxi_wj_wxi_sq_n}
u^{(j)}(\xi)=u(\xi), \quad w^{(j)}(\xi)=w(\xi), \quad \xi \in {\bf
V}([{\cal G}_{j,{\cal A}}]_{\le j-1});
\end{align}
\begin{align}
\label{ujxi} u^{(j)}(\xi)\underset{C_*}{\asymp} u_i\cdot
2^{\frac{\psi(N)-\psi(i)}{p}}, \quad
w^{(j)}(\xi)\underset{C_*}{\asymp} w_i\cdot
2^{-\frac{\psi(N)-\psi(i)}{q}}, \quad \xi \in {\bf V}^{{\cal
G}_{j,{\cal A}}}_{i}(\xi_0), \quad j\le i\le N;
\end{align}
if $C_*=1$, then we have exact equalities in (\ref{ujxi}).
\item If $j=0$, then ${\cal G}_{j,{\cal A}}$ is a disjoint union of
pathes $\zeta_{k,0}<\zeta_{k,1}<\dots <\zeta_{k,N}$,
$1\le k\le {\rm card}\, {\bf V}^{\cal A}_{N}(\xi_0)$. In addition,
\begin{align}
\label{u0zeta_ki}
u^{(0)}(\zeta_{k,i})\underset{C_*}{\asymp} u_i\cdot
2^{\frac{\psi(N)-\psi(i)}{p}}, \quad
w^{(0)}(\zeta_{k,i})\underset{C_*}{\asymp} w_i\cdot
2^{-\frac{\psi(N)-\psi(i)}{q}}, \quad 0\le i\le N.
\end{align}
If $C_*=1$, then we have exact equalities in (\ref{u0zeta_ki}).
\item $\mathfrak{S}^{p,q}_{{\cal A},u,w}\le
\mathfrak{S}^{p,q}_{{\cal G}_{j,{\cal A}},u^{(j)},w^{(j)}}$.
\end{enumerate}

The graphs ${\cal G}_{j,{\cal A}}$ and the functions $u^{(j)}$,
$w^{(j)}$ will be constructed by induction on $j$. Suppose that
for some $0\le k\le N-1$ the trees ${\cal G}_{k+1,{\cal A}}$ and
the functions $u^{(k+1)}$, $w^{(k+1)}$ are constructed, and
suppose that assertions 1--4 hold with $j:=k+1$. Set ${\bf
V}_k^{{\cal G}_{k+1,{\cal A}}}(\xi_0)=\{\zeta_1, \, \dots, \,
\zeta_m\}$. From (\ref{gjaj}) it follows that $\{\zeta_1, \,
\dots, \, \zeta_m\}={\bf V}_k^{{\cal A}}(\xi_0)$,
$u^{(k+1)}(\zeta_t)\stackrel{(\ref{uj_uxi_wj_wxi_sq_n})}{=}u(\zeta_t)$,
$w^{(k+1)} (\zeta_t) \stackrel{(\ref{uj_uxi_wj_wxi_sq_n})} {=}
w(\zeta_t)$, $1\le t\le m$.

We set
$$
{\cal G}_{k,{\cal A}}=G_{\zeta_m}(\dots G_{\zeta_2}(G_{\zeta_1}
({\cal G}_{k+1,{\cal A}}))),
$$
$$
u_k=(((u_{k+1})_{\zeta_1})_{\zeta_2}\dots)_{\zeta_m}, \quad
w_k=(((w_{k+1})_{\zeta_1})_{\zeta_2}\dots)_{\zeta_m}
$$
(see (\ref{u_xi_t_g_xi_t})). From Lemma \ref{razv} and
the induction assumption we obtain assertion 4. Conditions (\ref{gjaj}),
(\ref{cv1gj}), (\ref{cvvv}) for $j:=k>0$ and the first part of assertion
3 hold by construction and by the induction hypothesis.

Estimate the values $u^{(k)}(\eta)$ and $w^{(k)}(\eta)$, $\eta\in {\bf
V}({\cal G}_{k,{\cal A}})$. Let $\eta \in {\bf V}_k({\cal
G}_{k,{\cal A}})$. Then ${\bf V}_1^{{\cal G}_{k,{\cal
A}}}(\eta)=\{\eta'\}$. There exists $1\le t\le m$ such that $\eta'\in
{\bf V} ^{{\cal G}_{k+1,{\cal A}}}_1(\zeta_t)$. From definition of
$u^{(k)}$ and $w^{(k)}$ and from (\ref{uj_uxi_wj_wxi_sq_n})
applied to $j:=k+1$ we get
$$
u^{(k)}(\eta)\stackrel{(\ref{uxit_wxit_ch}),(\ref{cv1gj})}{=}
u^{(k+1)}(\zeta_t) \left({\rm card}\, {\bf V}^{\cal
A}_{N-k}(\zeta_t)\right)^{\frac 1p} \stackrel{(\ref{cvjj0})}
{\underset{C_*}{\asymp}} u(\zeta_t) 2^{\frac{\psi(N)-\psi(k)}{p}},
$$
$$
w^{(k)}(\eta)\stackrel{(\ref{uxit_wxit_ch}),(\ref{cv1gj})}{=}
w^{(k+1)}(\zeta_t) \left({\rm card}\, {\bf V}^{\cal
A}_{N-k}(\zeta_t)\right)^{-\frac 1q} \stackrel{(\ref{cvjj0})}
{\underset{C_*}{\asymp}} w(\zeta_t)
2^{-\frac{\psi(N)-\psi(k)}{q}}.
$$
If $C_*=1$, then we have exact equalities.

Let $\eta \in {\bf V}({\cal G}_{k,{\cal A}}) \backslash {\bf
V}_k({\cal G}_{k,{\cal A}})$. Then
$u^{(k)}(\eta)=u^{(k+1)}(\eta)$, $w^{(k)}(\eta)=w^{(k+1)}(\eta)$.
This together with the induction assumption yields
(\ref{uj_uxi_wj_wxi_sq_n}) and (\ref{ujxi}) for $k>0$ and
the second part of assertion 3 for $k=0$.

Let us estimate $\mathfrak{S}^{p,q}_{{\cal G}_{0,{\cal
A}},u^{(0)},w^{(0)}}$. Set
\begin{align}
\label{m_st} m_*={\rm card}\, {\bf V}_{N}^{\cal A}(\xi_0)
\stackrel{(\ref{cvjj0})}{\underset{C_*}{\asymp}} 2^{\psi(N)}
\end{align}
(if $C_*=1$, then the exact equality holds). By assertion 3,
$\mathfrak{S}^{p,q}_{{\cal G}_{0,{\cal A}},u^{(0)},w^{(0)}}
\underset{C_*}{\asymp} \mathfrak{S}^{p,q}_{{\cal G}_{0,{\cal
A}},\tilde u,\tilde w}$, where
\begin{align}
\label{tr_gle} \tilde u(\zeta_{k,i})=\tilde u_i:=u_i\cdot
2^{\frac{\psi(N)-\psi(i)}{p}}, \quad \tilde w(\zeta_{k,i})=\tilde
w_i:=w_i\cdot 2^{-\frac{\psi(N)-\psi(i)}{q}}.
\end{align}

Let $f:{\bf V}({\cal G}_{0,{\cal A}})\rightarrow \R_+$,
$\|f\|_{l_p({\cal G}_{0,{\cal A}})}=1$. Set $\varphi(\zeta_{k,i})
=\varphi_{k,i}=f^p(\zeta_{k,i})$. Then
\begin{align}
\label{fl_w_r} \sum \limits _{k=1}^{m_*} \sum \limits _{j=0}^N
\varphi_{k,i}=1,
\end{align}
$$
\sum \limits _{\xi\in {\bf V}({\cal G}_{0,{\cal A}})} \tilde
w^q(\xi) \left(\sum \limits _{\xi'\le \xi} \tilde u(\xi')
f(\xi')\right)^q= \sum \limits _{k=1}^{m_*} \sum \limits _{j=0}^N
\tilde w_j^q \left(\sum \limits _{i=0}^j \tilde
u_i\varphi_{k,i}^{1/p}\right)^q=:{\cal F}(\varphi).
$$
Since $p\ge q$, the function $t\mapsto t^{\frac qp}$ is concave on
$\R_+$. This together with
the inverse Minkowski inequality
implies that ${\cal F}(\varphi)$ is concave
on the set of nonnegative functions $\varphi$.

Set $\tilde \varphi(\zeta_{k,i})=\tilde \varphi_i=
\frac{1}{m_*}\sum \limits_{l=1} ^{m_*} \varphi_{l,i}$, $1\le k\le
m_*$, $\tilde f_i=\tilde \varphi_i^{1/p}m_*^{1/p}$. Then
\begin{align}
\label{fjp1} \sum \limits _{j=0}^N \tilde \varphi_i =
\frac{1}{m_*} \sum \limits _{k=1}^{m_*} \sum \limits _{j=0}^N
\varphi_{k,i}\stackrel{(\ref{fl_w_r})}{=}\frac{1}{m_*},\quad \quad \sum
\limits _{j=0}^N \tilde f_j^p =1.
\end{align}
Notice that $\tilde \varphi(\zeta_{k,i}) =\frac{1}{ {\rm card}\,
\mathbb{S}_{m_*}}\sum \limits _{\pi \in \mathbb{S}_{m_*}} \varphi
_{\pi}(\zeta_{k,i})$ and ${\cal F}(\varphi)={\cal
F}(\varphi_{\pi})$ for any $\pi \in \mathbb{S}_{m_*}$, where
$\mathbb{S}_{m_*}$ is the set of all permutations of $m_*$
elements and $\varphi_\pi(\zeta_{k,i})=\varphi(\zeta_{\pi(k),i})$.
Since ${\cal F}$ is concave, the inequality ${\cal F}(\varphi)\le
{\cal F}(\tilde \varphi)$ holds. Therefore,
$$
\sum \limits _{k=1}^{m_*} \sum \limits _{j=0}^N \tilde w_j^q
\left(\sum \limits _{i=0}^j \tilde
u_i\varphi_{k,i}^{1/p}\right)^q\le \sum \limits _{k=1}^{m_*} \sum
\limits _{j=0}^N \tilde w_j^q \left(\sum \limits _{i=0}^j \tilde
u_i\tilde \varphi_i^{1/p}\right)^q \stackrel{(\ref{tr_gle})}{=}
$$
$$
=m_*\sum \limits _{j=0}^N w_j^q\cdot 2^{-\psi(N)+\psi(j)}\left(
\sum \limits _{i=0}^j u_i \cdot 2^{\frac{\psi(N)-\psi(i)}{p}}
m_*^{-\frac 1p}\tilde f_i\right)^q \stackrel{(\ref{m_st})}
{\underset{C_*, p,q}{\asymp}}
$$
$$
\asymp \sum \limits _{j=0}^N w_j^q\cdot 2^{\psi(j)}\left( \sum
\limits _{i=0}^j u_i \cdot 2^{-\frac{\psi(i)}{p}} \tilde
f_i\right)^q \stackrel{(\ref{fjp1})}{\le}
\left[\mathfrak{S}^{p,q}_{\hat u,\hat w}\right]^q.
$$
This completes the proof.
\end{proof}
\renewcommand{\proofname}{\bf Proof}

The similar assertion can be obtained for the weighted integration
operator on a metric tree. Let $\mathbb{A}=({\cal A}, \, \Delta)$,
where $({\cal A}, \, \xi_0)$ satisfies (\ref{cvjj0}) and ${\rm
card}\, {\bf V}_1^{\cal A}(\xi_0)=1$. Suppose that $\Delta((\xi',
\, \xi'')) =[a_j, \, b_j]$ for any $\xi'\in {\bf V}^{\cal
A}_j(\xi_0)$, $\xi''\in {\bf V}_1^{\cal A}(\xi')$. Let $x_0$ be
the minimal point in $\mathbb{A}$. Consider the weight functions
$g$, $v:\mathbb{A} \rightarrow (0, \, \infty)$ such that
$g(x)=g_0(|x-x_0|_{\mathbb{A}})$, $v(x)=v_0(|x-x_0|_{\mathbb{A}})$
(see (\ref{yx_tt})).

Set $R=\sum \limits _{j\in \Z_+} (b_j-a_j)$,
$$
\hat v_0(t)=v_0(t)\cdot 2^{\frac{\psi(j)}{q}}, \quad
\hat g_0(t)=g_0(t)\cdot 2^{-\frac{\psi(j)}{p}}, \quad
t=\sum \limits _{i=0}^{j-1}(b_i-a_i) +s, \quad s\in [a_j, \, b_j].
$$
Let $I_{g,v,x_0}:L_p(\mathbb{A}) \rightarrow L_q(\mathbb{A})$ be
defined by (\ref{i_uw}), and let $\hat I_{\hat g_0,\hat
v_0}f(t)=\hat v_0(t)\int \limits_0^t \hat g_0(x)f(x)\, dx$, $0\le
t<R$, $f\in L_p(0, \, R)$.
\begin{Trm}
Let $1\le q\le p\le \infty$. Then
$\|I_{g,v,x_0}\|_{L_p(\mathbb{A}) \rightarrow L_q(\mathbb{A})}
\underset{p,q,C_*}{\asymp} \|\hat I_{\hat g_0,\hat v_0}\|_{L_p(0,
\, R)\rightarrow L_q(0, \, R)}$. If $C_*=1$, then the exact
equality holds.
\end{Trm}
This result is proved similarly as Theorem \ref{p_ge_q}.
For $p=q=2$ it was obtained in \cite{naim_sol}.

In conclusion, the author expresses her sincere gratitude to
V.D. Stepanov providing references.


\begin{Biblio}
\bibitem{and_hein} K.F. Andersen, H.P. Heinig, ``Weighted norm inequalities for
certain integral operators'', {\it SIAM J. Math. Anal.}, {\bf 14}
(1983), 834--844.

\bibitem{ben1}  G. Bennett, ``Some elementary inequalities'',
{\it Quart. J. Math. Oxford Ser. (2)}, {\bf 38}:152 (1987),
401–425.

\bibitem{ben2} G. Bennett, ``Some elementary inequalities. II'',
{\it Quart. J. Math. Oxford Ser. (2)} 39:156 (1988), 385–400.

\bibitem{bennett_g} G. Bennett, ``Some elementary inequalities.
III'', {\it Quart. J. Math. Oxford Ser. (2)}, {\bf 42}:166 (1991),
149–174.

\bibitem{j_brad} J.S. Bradley, ``Hardy inequalities with mixed
norms'', {\it Canad. Math. Bull.} {\bf 21}:4 (1978), 405–408.

\bibitem{vd_step94} M.Sh. Braverman, V.D. Stepanov, ``On the discrete Hardy
inequality'', {\it  Bull. London Math. Soc.}, {\bf 26}:3 (1994),
283--287.

\bibitem{evans_har} W.D. Evans, D.J. Harris, ``Fractals, trees and the Neumann
Laplacian'', {\it Math. Ann.}, {\bf 296}:3 (1993), 493--527.

\bibitem{e_h_l} W.D. Evans, D.J. Harris, J. Lang, ``Two-sided estimates for the approximation
numbers of Hardy-type operators in $L_\infty$ and $L_1$'', {\it
Studia Math.}, {\bf 130}:2 (1998), 171–192.

\bibitem{ev_har_lang} W.D. Evans, D.J. Harris, J. Lang, ``The approximation numbers
of Hardy-type operators on trees'', {\it Proc. London Math. Soc.}
{\bf (3) 83}:2 (2001), 390–418.

\bibitem{ev_har_pick} W.D. Evans, D.J. Harris, L. Pick, ``Weighted Hardy
and Poincar\'{e} inequalities on trees'', {\it J. London Math.
Soc.}, {\bf 52}:2 (1995), 121--136.

\bibitem{fars_sm} S.M. Farsani, ``On the boundedness and compactness of Riemann-Liouville fractional operators'' [Russian], {\it Sibirsk. Mat. Zh.}
{\bf 54}:2 (2013), 468–479.

\bibitem{gold_man} M.L. Goldman, ``Hardy type inequalities on the cone
of quasimonotone functions'', Research report 98/31, Russian Acad.
of Sciences, Far Eastern Branch, Khabarovsk, 1998.

\bibitem{grosse_erd} K.-G. Grosse-Erdmann, {\it The blocking technique, weighted
mean operators and Hardy's inequality}. Lecture Notes in
Mathematics, vol. 1679. Springer-Verlag, Berlin, 1998.


\bibitem{hein1} H.P. Heinig, ``Weighted norm inequalities for
certain integral operators, II'', {\it Proc. AMS}, {\bf 95}
(1985), 387--395.

\bibitem{kuf_mal_pers} A. Kufner, L. Maligranda, L.-E. Persson,
{\it The Hardy inequality. About its history and some related
results}. Vydavatelsky Servis, Plze\v{n}, 2007. 162 pp.

\bibitem{kuf_per} A. Kufner, L.-E. Persson, {\it Weighted inequalities of
Hardy type}. World Scientific Publishing Co., Inc., River Edge,
NJ, 2003.

\bibitem{l_leind}  L. Leindler, ``Generalization of inequalities of Hardy and
Littlewood'', Acta Sci. Math. {\bf 31} (1970), 279-285.

\bibitem{lifs_m} M.A. Lifshits, ``Bounds for entropy numbers for some critical
operators'', {\it Trans. Amer. Math. Soc.}, {\bf 364}:4 (2012),
1797–1813.

\bibitem{l_l} M.A. Lifshits, W. Linde, ``Compactness properties of weighted summation operators
on trees'', {\it Studia Math.}, {\bf 202}:1 (2011), 17--47.

\bibitem{l_l1} M.A. Lifshits, W. Linde, ``Compactness properties of weighted summation operators
on trees --- the critical case'', {\it Studia Math.}, {\bf 206}:1
(2011), 75--96.

\bibitem{mazya1} V.G. Maz’ja [Maz’ya], {\it Sobolev spaces} (Leningrad. Univ.,
Leningrad, 1985; Springer, Berlin–New York, 1985).

\bibitem{naim_sol} K. Naimark, M. Solomyak, ``Geometry of Sobolev spaces on regular trees and the Hardy
inequality'', {\it Russian J. Math. Phys.}, {\bf 8}:3 (2001),
322--335.

\bibitem{r_oin} R. Oinarov, ``Two-sided estimates for the norm of some classes of integral
operators'', {\it Trudy Mat. Inst. Steklov} {\bf 204} (1993),
240--250; translation in {\it Proc. Steklov Inst. Math.} {\bf 204}
(1994), 205–-214.

\bibitem{oin_per_tem} R. Oinarov, L.-E. Persson, A. Temirkhanova,
``Weighted inequalities for a class of matrix operators: the case
$p\le q$'', {\it  Math. Inequal. Appl.},  {\bf 12}:4 (2009),
891–903.

\bibitem{okp_per_wed1} C.A. Okpoti, L.E. Persson, A. Wedestig, ``Scales of weight characterizations
for some multidimensional discrete Hardy and Carleman type
inequalities'', {\it Proc. A. Razmadze Math. Inst.}, {\bf 138}
(2005), 63–84.

\bibitem{okp_per_wed2} C.A. Okpoti, L.E. Persson, A. Wedestig,
``Weight characterizations for the discrete Hardy inequality with
kernel'', {\it J. Inequal. Appl.}, 2006, Art. ID 18030, 14 pp.


\bibitem{pr_st} D.V. Prokhorov, V.D. Stepanov, ``Weighted estimates for Riemann -- Liouville operators and their applications'', {\it
Tr. Mat. Inst. Steklova} {\bf 243} (2003), 289--312; translation
in {\it Proc. Steklov Inst. Math.} {\bf 243} (2003), 278-–301.

\bibitem{n_raut} N.A. Rautian, ``On the boundedness of a class of fractional-type integral operators'', {\it Mat. Sb.}
{\bf 200}:12 (2009), 81--106; translation in {\it Sb. Math.} {\bf
200}:11-12 (2009), 1807–1832.

\bibitem{solomyak} M. Solomyak, ``On approximation of functions from Sobolev spaces on metric
graphs'', {\it J. Approx. Theory}, {\bf 121}:2 (2003), 199--219.

\bibitem{step90} V.D. Stepanov, ``Two-weight estimates for Riemann -- Liouville
integrals'',  {\it Izv. Akad. Nauk SSSR Ser. Mat.} {\bf 54}:3
(1990), 645--656; transl.: {\it Math. USSR-Izv.}, {\bf 36}:3
(1991), 669–681.

\bibitem{stepanov1} V.D. Stepanov, ``Two-weighted estimates for Riemann-Liouville
integrals'', Rept. 39, Ceskoslov. Akad. V\v{e}d. Mat. \'{U}stav.
Praha, 1988. P. 1--28.

\bibitem{st_ush} V.D. Stepanov, E.P. Ushakova, ``Kernel operators with variable intervals of integration in Lebesgue
spaces and applications'', {\it Math. Inequal. Appl.} {\bf 13}:3
(2010), 449–510.

\bibitem{vas_rjmp1} A.A. Vasil'eva, ``Embedding theorem
for weighted Sobolev classes with weights that are functions
of the distance to some $h$-set'', {\it Russ. J. Math. Phys.},
{\bf 20}:3, 360--373.

\bibitem{vas_rjmp2} A.A. Vasil'eva, ``Embedding theorem
for weighted Sobolev classes with weights that are functions
of the distance to some $h$-set. II'', {\it Russ. J. Math. Phys.},
to appear.
\end{Biblio}
\end{document}